\documentclass[paper=a4, fontsize=11pt,oneside]{scrartcl}	

\usepackage[a4paper,pdftex]{geometry}	
\usepackage{url}
\usepackage{float}
\usepackage[
	pdftex,
	pdfstartview={XYZ 0 1000 1.0},
	bookmarks=true,
	colorlinks,
	breaklinks=true,
	citecolor=black,
	filecolor=black,
	linkcolor=black,
	urlcolor=black,
	pdfborder={0 0 0}
]{hyperref}

\newcommand{\HRule}[1]{\rule{\linewidth}{#1}} 	

\makeatletter							
\def\printtitle{%
    {\centering \@title\par}}
\makeatother									

\makeatletter							
\def\printauthor{%
    {\centering \large \@author}}				
\makeatother							

\title{	
		 \normalsize \ \\[2.0cm]								
			\HRule{1pt} \\						
			\Large \textbf{\uppercase{PDE and hypersurfaces with prescribed mean curvature}}	
			\HRule{2pt} \\ [0.5cm]		
			\normalsize 
Lecture notes for the mini-course \textit{PDE and hypersurfaces with prescribed mean curvature} held in Federal University of S\~ao Carlos at the Workshop on Submanifold Theory and Geometric Analysis, August 05 -- 09, 2019.
}

\author{
{\Large		Yunelsy N\'apoles Alvarez}\\	
		IME-USP\\	
    \url{ynapolez@gmail.com} \\
		\url{ynalvarez@usp.br}\\
\small		This study was financed by the Coordena\c c\~ao de Aperfei\c coamento de Pessoal de N\'ivel Superior - Brasil (CAPES) -- Finance Code 001.\\
}

\usepackage[english]{babel}
\usepackage[T1]{fontenc}
\usepackage{lmodern}
\usepackage[utf8]{inputenc}

\usepackage[protrusion=true,expansion=true]{microtype}	
\usepackage{amsmath,amsfonts,amsthm,amssymb}
\usepackage{graphicx}
\usepackage[small]{caption}
\usepackage{bm}
\usepackage{makeidx}
\usepackage{mathtools}
\usepackage{indentfirst}
\usepackage{amssymb,amsmath,amsthm,amsfonts,amscd,amstext,amsbsy}
\usepackage{empheq}
\usepackage{latexsym}								   		
\usepackage{graphicx,graphics,epsfig}			
\usepackage{graphpap}										
\usepackage{float}
\usepackage{xcolor}
\usepackage{appendix}										
\usepackage{multicol}										
\usepackage{color}
\usepackage{epstopdf}		
\usepackage{enumerate}
\usepackage{wasysym}
\usepackage{mathrsfs}

\usepackage{lipsum}

\usepackage{mathtools}

\usepackage{enumitem}

\usepackage{makeidx}
\makeindex
\usepackage{mathtools}
\usepackage{indentfirst}
\usepackage{amssymb,amsmath,amsthm,amsfonts,amscd,amstext,amsbsy}

\usepackage{epic}
\usepackage{epsf}
\usepackage[active]{srcltx}

\usepackage{calligra}

\usepackage{empheq}
\usepackage{latexsym}				
\usepackage{graphicx,graphics,epsfig}			
\usepackage{graphpap}				
\usepackage{float}
\usepackage{xcolor}
\usepackage{color}
\usepackage{appendix}				
\usepackage{multicol}			
\usepackage{epstopdf}		      
\usepackage{enumerate}

\DeclareFontFamily{U}{matha}{\hyphenchar\font45}
\DeclareFontShape{U}{matha}{m}{n}{
      <5> <6> <7> <8> <9> <10> gen * matha
      <10.95> matha10 <12> <14.4> <17.28> <20.74> <24.88> matha12
      }{}
\DeclareSymbolFont{matha}{U}{matha}{m}{n}
\DeclareFontSubstitution{U}{matha}{m}{n}
\DeclareFontFamily{U}{mathx}{\hyphenchar\font45}
\DeclareFontShape{U}{mathx}{m}{n}{
      <5> <6> <7> <8> <9> <10>
      <10.95> <12> <14.4> <17.28> <20.74> <24.88>
      mathx10
      }{}
\DeclareSymbolFont{mathx}{U}{mathx}{m}{n}
\DeclareFontSubstitution{U}{mathx}{m}{n}
\DeclareMathDelimiter{\vvvert}{0}{matha}{"7E}{mathx}{"17}

\makeatletter
\def\namedlabel#1#2{\begingroup
   \def\@currentlabel{#2}%
   \label{#1}\endgroup
}
\makeatother

\newcommand{\nc}{\newcommand}

\nc{\ds}{\displaystyle}
\nc{\rarrow}{\rightarrow}
\nc{\lrarrow}{\longrightarrow}

\nc{\cl}{\mathscr{C}} 

\nc{\al}{\alpha}
\nc{\w}{\omega}
\nc{\W}{\Omega}

\newcommand{\norm}[1]{\left\|#1 \right\|}

\def\mod#1{\left|#1 \right|}
\def\modulo#1{\left|#1 \right|}
\newcommand{\escalar}[2]{{\left\langle #1,#2 \right\rangle}}

\newcommand{\parcial}[2]{\frac{\partial#1}{\partial#2}}

\newcommand{\parcialll}[3]{\frac{\partial^{2}#1}{\partial#2\partial#3}}

\newcommand{\Ei}{e_i}
\newcommand{\Ej}{e_j}
\newcommand{\Ek}{e_k}

\newcommand{\Et}{e_{n+1}}

\newcommand{\Dz}{\partial_z}

\newcommand{\Di}{\partial_i}
\newcommand{\Dj}{\partial_j}
\newcommand{\Dk}{\partial_k}

\newcommand{\Dij}{\partial_{ij}}

\newcommand{\Dii}{\partial_{ii}}
\newcommand{\Dkk}{\partial_{kk}}
\newcommand{\Dkki}{\partial_{kki}}
\newcommand{\Dkkum}{\partial_{kk1}}
\newcommand{\Dki}{\partial_{ki}}

\newcommand{\Dum}{\partial_1}
\newcommand{\Dumk}{\partial_{1k}}
\newcommand{\Dumj}{\partial_{1j}}

\newcommand{\Dumum}{\partial_{11}}
\newcommand{\Dumumum}{\partial_{111}}
\newcommand{\Dumij}{\partial_{1ij}}
\newcommand{\Dumii}{\partial_{1ii}}

\newcommand{\funciones}[5]{\begin{array}{rccl}
#1:&#2&\lrarrow&#3\\
&#4& \xmapsto{\phantom{\lrarrow}}&#5\end{array}}

\newcommand{\funcionesp}[5]{\begin{array}{rccl}
#1:&#2&\lrarrow&#3\\
&#4& \xmapsto{\phantom{\lrarrow}}&#5.\end{array}}

\usepackage{relsize}

\DeclareMathOperator{\Hess}{\nabla^2}
\newcommand{\Hessij}{\partial_{ij}}

\DeclareMathOperator{\diver}{div}
\DeclareMathOperator{\dist}{dist}

\DeclareMathOperator{\tr}{tr}
\DeclareMathOperator{\spann}{span}

\DeclareMathOperator{\diam}{diam}

\DeclareMathOperator{\R}{\mathbb{R}}
\DeclareMathOperator{\HH}{\mathbb{H}}

\newcommand{\Hc}{\mathcal{H}}

\DeclareMathOperator{\Ricc}{Ricc}

\newcommand{\LL}{\mathfrak{L}}

\DeclareMathAlphabet{\mathpzc}{OT1}{pzc}{m}{ }
\newcommand{\Q}{\mathfrak{Q}}

\newcommand{\M}{\mathcal{M}}

\theoremstyle{theorem}

\newtheorem*{teorema}{Theorem}

\newtheorem{taggedtheoremx}{Theorem}
\newenvironment{taggedtheorem}[1]
 {\renewcommand\thetaggedtheoremx{#1}\taggedtheoremx}
 {\endtaggedtheoremx}

\newtheorem{teo}{Theorem}[section]

\newtheorem{prop}[teo]{Proposition} 

\newtheorem{lema}[teo]{Lemma} 

\usepackage{amsthm}
\makeatletter
\def\th@plain{%
  \thm@notefont{}
  \itshape 
}
\def\th@definition{%
  \thm@notefont{}
  \normalfont 
}
\makeatother

\theoremstyle{definition}

\newtheoremstyle{remarknovo}
  {1.5\topsep}
  {1.5\topsep}
  {\normalfont}
  {}
  {\bfseries}
  {. }
  {0pt}
  {}
\theoremstyle{remarknovo}

\newtheorem{obs}[teo]{Remark}

\def\Hc{\mathcal H}

\def\rmd{\mathop{\rm d\kern -1pt}\nolimits}
\def\rme{\mathop{\rm e\kern -1pt}\nolimits}

\renewcommand{\div}{\text{div} \,}

\def \Hc{{\mathcal H}}

\def\bel{ \medskip
 \centerline{$ \ast \hbox to 1.0cm{}\ast \hbox to 1.0cm{}\ast $}
}

\let\leq=\leqslant
\let\geq=\geqslant

\newcounter{cpteur}
\makeatletter
\newcommand{\pushright}[1]{\ifmeasuring@#1\else\omit\hfill$\displaystyle#1$\fi\ignorespaces}
\newcommand{\pushleft}[1]{\ifmeasuring@#1\else\omit$\displaystyle#1$\hfill\fi\ignorespaces}
\makeatother

\makeatletter
\DeclareRobustCommand*{\bfseries}{%
  \not@math@alphabet\bfseries\mathbf
  \fontseries\bfdefault\selectfont
  \boldmath
}
\makeatother

\colorlet{shadecolor}{gray!40}

\RequirePackage{todonotes}

\begin{document}
\thispagestyle{empty}		

\printtitle					
  	\vfill
\printauthor				
\newpage
\setcounter{page}{1}		


\section{Introduction}\label{intro}

The aim of these notes is to introduce to the geometers useful tools from the \textit{Theory of partial differential equations} which are used in order to obtain hypersurfaces with prescribed mean curvature. For instance, graphs with prescribed mean curvature can be obtained by solving the Dirichlet problem for a particular quasilinear elliptic partial differential equation of second order. 


The Dirichlet problem for the prescribed mean curvature equation consists on find a function satisfying the prescribed mean curvature equation in a bounded domain of the $n-$dimensional Euclidean ambient space and that continuously takes on given boundary values. 
Precisely, given a smooth bounded domain $\W\subset\R^n$ ($n\geq 2$) and $\varphi\in\cl^0({\partial\W})$, we ask if for a prescribed smooth function {$H$} there exists some $u\in\cl^2(\W)\cap\cl^0(\overline{\W})$ satisfying
\begin{equation}\tag{$P$}\label{ProblemaP}
\left\{
\begin{split}
\diver \left(\dfrac{\nabla u}{\sqrt{1+\norm{\nabla u}^2}} \right)&=n H(x) \ \mbox{in}\ \W,\\
u&=\varphi \ \mbox{in}\ \partial\W.
\end{split}\right.
\end{equation} 
If this is the case, then the graph of $u$ 
 is an hypersurface in $\R^{n+1}$ of mean curvature $H(x)$ at each point $(x,u(x))$. 

This problem traces its roots to the Plateau's problem that was first posed by Lagrange \cite{lagrange} in 1760. Lagrange wanted to find the surface with the least area among all the surfaces having the same boundary. This problem arose as an example of the Calculus of Variations that he was developing. In order to find a minimum of the area functional, Lagrange derived the Euler-Lagrange equation for the solutions of this problem. That is, if an area minimizing surface in the three dimensional Euclidean space is a graph of a smooth function $u$ over a bounded domain, then $u$ necessarily satisfies 
\begin{equation}\label{eq_minima_n2}
\left(1+u_y^2 \right)u_{xx} - 2u_xu_yu_{xy}\left(1+u_x^2 \right)u_{yy}=0 
\end{equation}
in that domain. 

It was Meusnier \cite{Meusnier} in 1776 who gave a geometrical interpretation for equation \eqref{eq_minima_n2}. He realized that the connexion between this partial differential equation and Geometry is given by the concept of \textit{mean curvature} that was formally introduced on his work. Indeed, the mean curvature of the graph of the solutions of this equation must vanishes. All those surfaces are called minimal surfaces even though a surface having vanish mean curvature is not necessarily globally area minimizing. 

Find examples of functions satisfying the minimal surface equation \eqref{eq_minima_n2} is not an easy task. As a matter of fact, Lagrange could only give the constants functions as examples. It was Meusnier \cite{Meusnier} in 1776 who gave another example of such a function that locally represents the helicoid. He also proved that the catenoid, discovered by Euler in 1741, can be seen locally as the graph of a function satisfying \eqref{eq_minima_n2}. In 1834, almost seventy years later of the discovering of the catenoid and the helicoid, H. F. Scherk \cite{scherk} gave new examples of minimal surfaces, the most famous one is the Sherk's surface. 

At this point, the interest was about the geometry of the domains over which minimal graphs can actually exist rather than find explicit expressions of the solutions of \eqref{eq_minima_n2}.  In 1910, Bernstein \cite{Bernstein} showed that the Dirichlet problem for equation \eqref{eq_minima_n2} has solution in disks of $\R^2$ for continuous boundary data. But Bernstein also realized that the disks could be replaced by convex domains (see \cite[p. 236]{Bernstein}). Proofs of this fact were given independently in 1930 by Douglas \cite{Douglas1931} and Radó \cite[p. 795]{Rado1930} who also ensured the uniqueness of the solution. In 1965 Finn \cite{Finn1965} made an important contribution when he proved that the Dirichlet problem for equation \eqref{eq_minima_n2} may not be solvable if the domain is non-convex. That is, this convexity conditions is sharp in order to obtain minimal graphs for arbitray continuous boundary values. 
Using the results from Douglas and Radó, and following a suggestion from Osserman, Finn stated the following {sharp} theorem:
\begin{taggedtheorem}{A}[Douglas-Rad\'o-Finn {\cite[T. 4a p. 146]{Finn1965}}]\label{SharpFinn}
The Dirichlet problem for equation \eqref{eq_minima_n2} has solution in $\W$ for arbitrary continuous boundary data if, and only if, $\W$ is convex. 
\end{taggedtheorem}

In 1966, Jenkins and Serrin \cite{Serrin1968} generalized Theorem \ref{SharpFinn} to higher dimensions: 
\begin{taggedtheorem}{B}[Jenkins-Serrin {\cite[T. 1 p. 171]{Serrin1968}}]\label{SharpJenkinsSerrin}
Let $\W$ be a bounded domain in $\R^n$ whose boundary is of class $\cl^2$. Then the Dirichlet problem for the minimal surface equation ($H=0$) in $\W$ is uniquely solvable for arbitrary $\cl^2$ boundary values if, and only if, the mean curvature of $\partial\W$ is non-negative. 
\end{taggedtheorem}
It can be observed that, although the convexity of $\W$ is not the appropriated generalization of the two-dimensional case (as it was thought), it is again a geometric property. We recall that a domain whose boundary has non-negative mean curvature is called a \textit{mean convex} domain. Throughout the text $\Hc_{\partial\W}$ will denote the mean curvature of $\partial\W$. 



Mathematicians were also questioning about the existence of graphs whose mean curvature not necessarily vanishes. It was Serrin \cite{Serrin} who gave a complete answer: 
\begin{taggedtheorem}{C}[Serrin {\cite[p. 416]{Serrin}}]\label{SharpSerrin} 
Let $\W$ be a bounded domain in $n$-dimensional Euclidean space whose boundary is of class $\cl^2$. Then for every constant $H$ the Dirichlet problem \eqref{ProblemaP} has a unique solution for arbitrary $\cl^2$ boundary data if, and only if, 
\begin{equation}\label{cond_serrin_constante}
(n-1) \Hc_{\partial\W}(y) \geq n H \ \ \forall\ y\in\partial\W.
\end{equation}
\end{taggedtheorem}


This work of Serrin was actually focused on the study of a more general class of Dirichlet problems within which is problem \eqref{ProblemaP}. In fact, Theorem \ref{SharpSerrin} is a direct conclusion of the following result: 
\begin{taggedtheorem}{D}[Serrin {\cite[T. p. 484]{Serrin}}]\label{T_Serrin_Ricci}
Let $\W$ be a bounded domain in $n$-dimensional Euclidean space, whose boundary is of class $\cl^2$. 
Let $H\in\cl^{1}(\overline{\W})$ and suppose that 
\begin{equation}\label{cond_Ricc_Serrin}
\norm{\nabla H(x)}\leq \dfrac{n}{n-1}(H(x))^2\ \forall\ x\in\W.
\end{equation}
Then problem \eqref{ProblemaP} is uniquely solvable for arbitrarily given $\cl^2$ boundary values if, and only if, 
\begin{equation}\label{SerrinCondition}
(n-1)\Hc_{\partial\W}(y)\geq n\mod{H(y)} \ \forall \ y\in\partial\W.
\end{equation}
\end{taggedtheorem}

Since the non-divergence form of the minimal operator is a quasilinear elliptic operator of second order, the basic tools for the solvability of problem \eqref{ProblemaP} come from the theory of partial differential equations. However, it can be seen in Theorem \ref{SharpSerrin} how the existence of graphs with constant mean curvature imposes, as in the minimal case, a geometric restriction over the domain. In the more general case, the existence of a function $H$ satisfying the hypothesis of Theorem \ref{T_Serrin_Ricci} leads to {additional geometric implications on $\W$}. 
This relation between the solvability of problem \eqref{ProblemaP} and the geometry of the domain is due to the fact that the minimal operator is non-uniformly elliptic. 


In order to solve problem \eqref{ProblemaP} the Leray-Schauder fixed point theorem is used in these notes. However, the beautiful Continuity Methods is another tool that can be used for the same purpose. In any case, the application of the method strongly depends on the existence of a priori estimates for the solutions of some ``related problems''. 

These notes are organized as follows. In section \ref{apendice_Anal_equac} is obtained the non-divergence form of the equation of the prescribed mean curvature and some analytical properties about the minimal operator are stated. 
Section \ref{chapterExistence} treats about the existence of solutions of problem \eqref{ProblemaP}. The estimates needed for the existence program are stated in section \ref{chapter_estimates}. The sharpness of the Serrin condition is proved in section \ref{cap_NaoExis}. With the intention of having an understandable text, some of the results from the theory of partial differential equations needed were established in the placed they are used. {At the end of the text it can be found some recent results in more general ambient spaces. 

Finally, we want to point out that these notes definitely do not represent the whole subject. For instance, some regularity issues are left to the reader as a motivation for further studies.}

\section{Analysis of the mean curvature equation}\label{apendice_Anal_equac}

\subsection{The non-divergence form of the mean curvature equation}

Let $S$ be an oriented hypersurface in $\R^{n+1}$. First, recall that if $N$ is a normal vector field along $S$, then the \textit{principal curvatures} of $S$ at a point $p\in S$ are the eigenvalues of the Weingarten map (or shape operator) $A_N$ at $p$ and the \textit{mean curvature} $H$ at $p$ is the average of the principal curvatures at $p$. Therefore,
\begin{equation}\label{curv_media_metrica}
nH=\tr\left(A_N\right)=\ds\sum_{i,j=1}^{n-1} g^{ij}b_{ij}, 
\end{equation}
where $g$ is the induced metric on $S$ and $b_{ij}$ are the coefficients of the second fundamental form\footnote{Recall that the second fundamental form is a symmetric bilinear form and that the Weingarten map is the self-adjoint linear transformation associated. Recall also that if we replace $N$ by $-N$, the principal curvatures change the sign, but the corresponding principal directions remain the same.}. 


In the case where $S$ is the graph over a domain $\W\subset M$ of a function $u\in\cl^2(\W)$ the map
$$ \funcionesp{F}{\W}{\R^{n+1}}{x}{(x,u(x))}$$
is a global parametrization of $S$. The tangents vectors on $S$ are 
$$X_i=\parcial{F}{x_i}=\Ei+ \Di u \Et, \ 1\leq i \leq n,$$ 
where $\{e_1,\dots, e_{n+1}\}$ is the canonical basis on $\R^{n+1}$. Hence, the coefficients of the induced metric $g$ on $S$ and of the inverse matrix $g^{-1}$ are given, respectively, by 
\[    g_{ij}=\escalar{X_i}{X_j}=\delta_{ij}+ \Di u\Dj u\]
and
\begin{equation}\label{inv_metr_S}
    g^{ij}=\delta_{ij}-  \frac{\Di u\Dj u}{W^2}.
\end{equation}



Besides, the unit normal field along $S$ is
\[N=\dfrac{-\nabla u (x) + \Et}{W},    \]
where (and from now on) 
\[W=\sqrt{1+\norm{\nabla u (x)}^2}.\]
Therefore, once
$$\parcialll{F}{x_i}{x_j}=\Dij u e_{n+1},$$ 
the expression for the coefficients of the second fundamental form is
\begin{equation}\label{bij_eq_curv_media_graf_MxR}
b_{ij}=II(X_i,X_j)=\escalar{N}{\parcialll{F}{x_i}{x_j}}=\frac{1}{W}\Dij u.
\end{equation}

Using \eqref{inv_metr_S} and \eqref{bij_eq_curv_media_graf_MxR} in \eqref{curv_media_metrica} it follows that $u$ satisfies the equation
\begin{equation}\label{operador_minimo_1_anal_eq}
\ds\frac{1}{W}\sum_{i,j=1}^n \left(\delta_{ij} - \frac{\Di u \Dj u}{W^2} \right)\Dij u=nH.
\end{equation}
Finally, some algebraic computations show that 
$$n H = \div\left(\dfrac{\nabla u}{\sqrt{1+\norm{\nabla u}^2}}\right). $$

\subsection{Ellipticity of the minimal operator}\label{secao_eq_cur_media}

Note that equation \eqref{operador_minimo_1_anal_eq} is equivalent to
\begin{equation}\label{operador_minimo_1_coord}
\M u:=\ds\sum_{i,j=1}^n \left(W^2\delta_{ij} - {\Di u \Dj u} \right)\Dij u=nH(x){W^3}. 
\end{equation}
We call $\M$ of minimal operator. From now on will also be considered the equation 
\begin{equation}\label{operador_Q}
\Q u:=\M u-nH(x)W^3=0.
\end{equation}

The operator $\M$ is a quasilinear operator of second order. In fact the coefficients of the second order derivatives are given by
\begin{equation}\label{coeff_aij}
a_{ij}(p)=\left((W(p))^2\delta_{ij}-{p_i p_j}\right),
\end{equation}
where $p$ stands for $\nabla u$ and $W(p)=\sqrt{1+\norm{p}^2}$. 
The aim of this section is to prove that the minimal operator $\M$ is strictly elliptic. That is, $A(p)=(a_{ij}(p))$ is positive definite for each fixed $p$, and the infimum of the smallest eigenvalue $\lambda(p)$ of $A(p)$ is positive\footnote{For a precise definition of ellipticity and strict ellipticity we refer that in \cite[p. 259]{GT}. Observe that in the case of the minimal operator $\M$ the matrix $A$ does not depend on $x$.}.  
This is an important property which is required in many of the results from the theory of partial differential equations used in this text. 

In order to prove that $A(p)$ is positive definite, it is just to find the eigenvalues of 
the quadratic form 
\begin{align*}
Q_p(q)=\escalar{{A}(p) q}{q}&=\sum_{i,j=1}^n\left((W(p))^2\delta_{ij}-p_i p_j\right)q_iq_j =\left(1+\norm{p}^2\right)\norm{q}^2-\escalar{p}{q}^2.
\end{align*}
If $q=p$ it follows 
\begin{align*}
\escalar{{A}(p) p}{p}&=\left(1+\norm{p}^2\right)\norm{p}^2-\norm{p}^4=\norm{p}^2, 
\end{align*}
and for $q\bot p$ 
\begin{align*}
\escalar{{A}(p) q}{q}&=\left(1+\norm{p}^2\right)\norm{q}^2.
\end{align*}
Therefore, the smallest eigenvalues of ${A}(p)$ is $\lambda(p)=\lambda= 1$ with associated eigenspace $\spann\{p\}$. Consequently, $\M$ and $\Q$ are strictly elliptic operators. 

Observe also that $\Lambda(p)=1+\norm{p}^2$ is the largest {(and there is no other)} eigenvalue of ${A}(p)$ whose associated eigenspace is ${p}^{\bot}$. Since $\Lambda(p)\rightarrow\infty$ as $\norm{p}\rightarrow\infty$, $\M$ and $\Q$ are non-uniformly elliptic operators. {This is the reason why the geometry of the domain is important for the solvability of problem \eqref{ProblemaP} (see \cite[\S 12.4 p. 309, p. 345]{GT}).}

\subsection{Maximum Principles}
{
The maximum principles are important tools used in the study of second order elliptic equations. The following theorem is restricted to the particular case of the operator $\Q$ defined in \eqref{operador_Q}.
\begin{teo}[Comparison principle {\cite[Th. 10.1 p. 263]{GT}}]\label{PM_quasilineares}
Let $\W\in\R^n$ be a bounded domain and $u,\ v\in\cl^2(\W)\cap\cl^0(\overline{\W})$ satisfying 
$$\left\{
\begin{split}
    \Q u&\geq \Q v \mbox{ in } \W,\\
     u&\leq v  \mbox{ in } \partial\W.
\end{split}\right.$$
Then $u\leq v$ in $\W$. 
\end{teo}
\begin{proof} Let $w=u-v$. 
Recalling the expression for the coefficients $a_{ij}$ given in \eqref{coeff_aij} it follows
\begin{align*}
    \Q u - \Q v
    =&\ds\sum_{ij}a_{ij}(\nabla u)\Dij w+\ds\sum_{ij}a_{ij}(\nabla u)\Dij v-\ds\sum_{ij}a_{ij}(\nabla v)\Dij v.
\end{align*}
For each $x\in\W$ let us define 
$$f_x(p)=\ds\sum_{ij}a_{ij}(p)D_{ij}v(x)-nH(x)(W(p))^3.$$
Applying the mean value theorem to the function
$$\varphi_x(t)=f((1-t)\nabla u(x)+t\nabla v(x)),$$
we obtain
\begin{align*}
\Q u-\Q v = & \ds\sum_{ij}a_{ij}(\nabla u(x))\Dij w + \ds\sum_{i}b_i(x)\Di w,
\end{align*}
where
\begin{align*}
{b}_i(x)=&\ds\sum_{kl} \dfrac{\partial a_{kl}}{\partial p_i}(x,(1-t(x))\nabla u(x)+t(x)\nabla v(x))D_{ij}v(x)\\&-nH(x)\dfrac{\partial \left(W^3\right)}{\partial p_i}(x,(1-t(x))\nabla u(x)+t(x)\nabla v(x)).
\end{align*}
So, $\LL w:=\Q u-\Q v\geq 0$. Since also $w\leq 0$ in $\partial\W$, then $w\leq 0$ in $\W$ as a direct consequence of the following theorem for linear operators.

\medskip
\noindent\colorbox{shadecolor}{
\begin{minipage}{.975\textwidth}
\begin{teorema}[Week maximum principle {\cite[Th. 3.1 p. 32]{GT}}]\label{WMPrinciple}
Let $\Omega\subset\mathbb{R}^n$ be a bounded domain. Assume that $u\in\cl^2(\W)\cap\cl^0(\overline{\W})$ satisfies 
$$\LL u=\ds\sum_{ij}a_{ij}(x)\Dij u+\ds\sum_{i}b_i(x)\Di u \geq 0 \mbox{ in }\W,$$ 
where $\LL$ is elliptic and the coefficients $a_{ij}$ and $b_i$ are locally bounded. 
Then
$$ \sup_{\W} u =\sup_{\partial\W}u.$$
\end{teorema}
\end{minipage}}
\medskip

\noindent Indeed, it can easily be verified that $\LL$ satisfies the hypothesis of this theorem. 
\end{proof}
}

The following proposition is an important tool in section \ref{cap_NaoExis}. 
This variant of the comparison principle traces its roots back to the work of Finn \cite[Lemma p. 139]{Finn1965} about the minimal surface equation in domains of the plane. 
His lemma was extended by Jenkins-Serrin {\cite[Prop. III p. 182]{Serrin1968}} for the minimal hypersurface equation in $\R^n$, and subsequently by Serrin \cite[Th. 1 p. 459]{Serrin} for more general quasilinear elliptic operators (see also \cite[Th. 14.10 p. 347]{GT}). {We just refer here to the operator $\Q$ defined in \eqref{operador_Q}.}
\begin{prop}[{\cite[L. 3.4.3 p. 109]{han2016nonlinear}}]\label{M_prop_gen_JS}
Let $\W\in M$ be a bounded domain and $\Gamma$ a relative open portion of $\partial\W$ of class $\cl^1$. For $H\in\cl^{0}(\overline{\W})$, let $u\in\cl^2(\W)\cap\cl^1(\W\cup \Gamma)\cap\cl^0(\overline{\W})$ and $v\in\cl^2(\W)\cap\cl^0(\overline{\W})$ satisfying
$$
\left\{ \begin{array}{cl}
\Q u \geq \Q v &\mbox{in } \W,\\
   u \leq v &\mbox{in } \partial\W\setminus\Gamma,\\
	\parcial{v}{N}=-\infty &\mbox{in }\Gamma,
\end{array} \right.
$$
where $N$ is the inner unit normal to $\Gamma$. Then $u\leq v$ in $\Gamma$. Consequently, $u\leq v$ in $\W$.
\end{prop}
\begin{proof}
Suppose by contradiction that $m=\ds\max_{\Gamma} (u-v) > 0.$
Hence, $u \leq v + m $ in $\Gamma$. Then $u\leq v+m$ in $\partial\W$ since $u\leq v$ in $\partial\W\setminus\Gamma$ by hypotheses.
Also,
$$\Q(v+m)=\Q v \leq \Q u.$$
As a consequence of Theorem \ref{PM_quasilineares}, $u\leq v+m$ in $\W$.

Let now $y_0\in \Gamma$ be such that $m = u(y_0)-v(y_0)$. 
Then, for $t>0$ near $0$ one has
\begin{align*}
u(y_0+tN_{y_0})-u(y_0)
\leq	 \left(v\left(y_0+tN_{y_0}\right)+m\right)-\left(v(y_0)+m\right)= v(y_0+tN_{y_0})-v(y_0).
\end{align*}
Dividing the expression by $t$ and passing to the limit as $t$ goes to zero it follows that 
$$ \parcial{u}{N}\leq \parcial{v}{N} =-\infty,$$ 
which is a contradiction since $u\in\cl^1(\Gamma)$. Hence, $u\leq v$ in $\Gamma$. Using again Theorem \ref{PM_quasilineares} one has $u\leq v$ in $\W$. 
\end{proof}



\subsection{Transformation formulas for the mean curvature equation}\label{apend_form_transform}

The goal in this section is to derive some transformation formulas that are used in the next sections. Note first that the operator $\M$ defined in \eqref{operador_minimo_1_coord} can be written as 
\begin{equation}\label{equ_curv_media_MxR_Delta_2}
\M u
=W^2\Delta u(x) - \escalar{\Hess  u \cdot \nabla u}{\nabla u},
\end{equation}
where $\Hess$ denotes the Hessian of a function (the matrix of second derivatives). 


Let $\W\in \R^n$ be a bounded domain and $\varphi\in\cl^2(\overline{\W})$. Let $I$ be an interval and $\psi\in\cl^2(I)$. 
Let us now define
$$w=\psi\circ \varrho + \varphi,$$ 
where $\varrho$ is a distance function\footnote{Recall that if $\varrho:\W\rightarrow \R$ is a distance function, then $\nabla \varrho \equiv 1$ on $\W$ \cite[p. 41]{petersen1998}.}.  
Then, $\Di w=\psi'\Di\varrho+\Di\varphi$ and $\nabla w=\psi'\nabla \varrho+\nabla \varphi$. Thus,
\begin{align}
\Hessij w &=\psi'' \Di\varrho\Dj\varrho   + \psi' \Hessij \varrho + \Hessij \varphi\label{hess_coord_1},
\end{align}
and 
\[\Hess  w=\psi'' \nabla \varrho \otimes \nabla \varrho + \psi' \Hess \varrho +\Hess \varphi .
\]
Since $\varrho$ is a distance function it follows
\begin{equation}\label{Laplaciano_w}
\Delta w(x)=\tr \left(\Hess  w\right) =\psi''\norm{\nabla \varrho}^2+\psi'\Delta \varrho + \Delta\varphi=\psi''+\psi'\Delta \varrho + \Delta\varphi.
\end{equation}
Besides\footnote{Recall that $\left(u\otimes v\right)w=\escalar{v}{w}u$ for any vectors $u$, $v$ and $w$.},
\begin{align*}
&\escalar{\Hess  w \cdot \nabla w}{\nabla w} \\
=&\psi'' \escalar{\nabla  \varrho}{\nabla w}^2  + \psi' \escalar{\Hess \varrho\cdot \nabla w}{\nabla w}+\escalar{\Hess \varphi \cdot \nabla w}{\nabla w}\\
=&\psi'' \escalar{\nabla  \varrho}{\psi' \nabla\varrho + \nabla \varphi}^2  + \psi' \escalar{\Hess \varrho \cdot \left(\psi' \nabla\varrho + \nabla \varphi\right)}{ \psi' \nabla\varrho + \nabla \varphi}\\
&+\escalar{ \Hess \varphi\cdot\left(\psi' \nabla\varrho + \nabla \varphi\right)}{\psi' \nabla\varrho + \nabla \varphi}\\
=&\psi'' \escalar{\nabla  \varrho}{\psi' \nabla\varrho + \nabla \varphi}^2\\
&+\psi'\left(\psi'^2 \escalar{ \Hess \varrho \cdot \nabla\varrho}{\nabla\varrho} + 2\psi' \escalar{\Hess \varrho\cdot\nabla\varrho}{\nabla \varphi} + \escalar{\Hess \varrho\cdot\nabla \varphi}{\nabla \varphi}\right)\\
&+\escalar{\Hess \varphi\cdot\left(\psi' \nabla\varrho + \nabla \varphi\right)}{\psi' \nabla\varrho + \nabla \varphi}.
\end{align*}
{Recalling again that $\varrho$ is a distance function one has
\begin{align*}
\escalar{\Hess \varrho\cdot \nabla \varrho}{e_j}=\ds\sum_{i} \Dij \varrho \Di \varrho =\escalar{\Dj \nabla\varrho}{\nabla\varrho} = \frac{1}{2} \Dj \left(\norm{\nabla \varrho }^2\right)=0\ \ \forall \ 1\leq j \leq n. 
\end{align*}}
As a consequence,
\begin{equation}\label{D2w}
\begin{array}{r}
\escalar{\Hess  w \cdot \nabla w}{\nabla w} =  \psi'' \escalar{\nabla  \varrho}{\psi' \nabla\varrho + \nabla \varphi}^2 + \psi' \escalar{\Hess \varrho \cdot \nabla \varphi}{\nabla \varphi}\\
+ \escalar{\Hess \varphi \cdot \left(\psi' \nabla\varrho + \nabla \varphi\right)}{\psi' \nabla\varrho + \nabla \varphi}.
\end{array}
\end{equation}

{From \eqref{equ_curv_media_MxR_Delta_2}, \eqref{Laplaciano_w} and \eqref{D2w} we derive}
\begin{align*}
\M w=&W^2(\psi''+\psi'\Delta \varrho + \Delta\varphi)-\left(\psi''\escalar{\nabla  \varrho}{\psi' \nabla\varrho + \nabla \varphi}^2 + 
\psi'\escalar{\Hess \varrho\cdot \nabla \varphi}{\nabla \varphi}\right.\\
&\pushright{\left.+ \escalar{\Hess \varphi\cdot \left(\psi' \nabla\varrho + \nabla \varphi\right)}{\psi' \nabla\varrho + \nabla \varphi}\right),} 
\end{align*}
where $W=W(\nabla w)=\sqrt{1+\norm{\psi'\nabla \varrho +\nabla \varphi}^2}$. Therefore, 
\begin{equation}\label{Eq_trans_varphi}
\begin{split}
\M w=&\psi'W^2\Delta\varrho -\psi'\escalar{\Hess \varrho\cdot\nabla \varphi}{\nabla \varphi}  \\
&\psi''W^2-\psi''\escalar{\nabla  \varrho}{\psi' \nabla\varrho + \nabla \varphi}^2\\
& \Delta \varphi W^2   - \escalar{\Hess \varphi\cdot\left(\psi' \nabla\varrho + \nabla \varphi\right)}{\psi' \nabla\varrho + \nabla \varphi} .
\end{split}
\end{equation}

Furthermore, if $\varphi$ is constant, then $W=\sqrt{1+\psi'^2}$ and 
\begin{equation}\label{calc_Q_w}
\M w =  \psi'(1+\psi'^2) \Delta \varrho + \psi''.
\end{equation}


\section{The existence program}\label{chapterExistence}

In this section we carry out with the existence program for the Dirichlet problem \eqref{ProblemaP}. It will be seen how the elliptic theory assures that the solvability of problem \eqref{ProblemaP} strongly depends on $\cl^{1}$ a priori estimates for the family of related problems 
\begin{equation}\tag{$P_{\tau}$}\label{ProblemaPsigma}
\left\{\begin{array}{l}
 \M u  = \tau nH(x)W^3\ \mbox{ in }\ \W, \\
 \phantom{\M}      u  = \tau \varphi \ \mbox{ in }\  \partial\Omega,
\end{array}\right.
\end{equation}
not depending on $\tau$ or $u$\footnote{Note that problem \eqref{ProblemaP} is equivalent to problem $(P_{\tau=1})$.}. Actually, the following theorem (which is also valid for more general elliptic operators) holds.
\begin{teo}[{\cite[T. 11.4 p. 281]{GT}}]\label{T_Exist_quaselineares}
Let $\Omega\subset \R^n$ be a bounded domain with $\partial\W$ of class $\cl^{2,\alpha}$, $\varphi\in \cl^{2,\alpha}(\overline{\Omega})$ and {$H\in\cl^{\alpha}(\overline{\W})$} for some $\alpha\in(0,1)$. Assume there exists $M>0$ independent of $u$ and $\tau$ such that any solution $u$ of the related problems \eqref{ProblemaPsigma} satisfies $\norm{u}_{\cl^{1}(\overline{\W})}<M$. 
Then the Dirichlet problem \eqref{ProblemaP} has a unique solution in $\cl^{2,\alpha}(\overline{\W})$.
\end{teo}
\begin{proof}
Let $\beta\in[0,1]$ to be made precise later. 
For each $v\in \cl^{1,\beta}(\overline{\Omega})$ we consider the linear operator 
\begin{equation}\label{definitionLL}
\LL^v u: =\ds\sum_{ij} a_{ij}^v(x)\Dij u ,
\end{equation}
where
$$ a_{ij}^v(x)=({W_v(x)})^2\delta_{ij}-{\Di v(x) \Dj v(x)},\ W_v(x)=W(\nabla v(x))=\sqrt{1+\nabla v(x)}.$$

We now want to study the following linear problem
\begin{equation}\tag{$LP_v$}\label{ProblemaPv}
\left\{\begin{array}{rl}
\LL^v u & = nH(x)({W_v(x)})^3\ \mbox{ in }\ \W, \\
      u & = \varphi \ \mbox{ in }\  \partial\Omega.
\end{array}\right.
\end{equation}

The following theorem for linear operators guaranties the existence of a unique solution $u^v\in\cl^{2,\alpha\beta}(\overline{\W})\subset\cl^{1,\beta}(\overline{\W})$ of problem \eqref{ProblemaPv}. 

\medskip
\noindent\colorbox{shadecolor}{
\begin{minipage}{.97\textwidth}
\begin{teorema}[Schauder {\cite[T. 6.14 p. 107]{GT}}]
Let $\Omega\subset\mathbb{R}^n$ be a $\cl^{2,\alpha}$ bounded domain and $\varphi\in \cl^{2,\alpha}(\overline{\Omega})$ for some $\alpha\in(0,1)$. Let 
$\LL u=\ds\sum_{ij}a_{ij}(x)\Dij u+\ds\sum_{i}b_i(x)\Di u +c(x)u$ 
be a strictly elliptic operator in $\W$ with coefficients in $\cl^{\alpha}(\overline{\W})$ and $c\leq 0$. Let also $f$ be a function in $\cl^{\alpha}(\overline{\W})$. Then the problem 
\begin{align*}
\left\{\begin{array}{l}
\LL u=f \textrm{ in } \Omega,\\
\phantom{\LL} u=\varphi\textrm{ in } \partial\Omega,
\end{array}\right.
\end{align*}
has a unique solution in $\cl^{2,\alpha}(\overline{\Omega})$.
\end{teorema}
\end{minipage}}
\medskip

In fact, notice that $\partial\W$ is of class $\cl^{2,\alpha\beta}$ and $\varphi\in\cl^{2,\alpha\beta}(\overline{\W})$. 
Also, $nHW_v^3$ and $a_{ij}^v$ belong to $\cl^{\alpha\beta}(\overline{\W})$ since $v\in\cl^{1,\beta}(\overline{\W})$, $H\in\cl^{\alpha}(\overline{\W})$ and the coefficients $a_{ij}$ are regular. 
Besides, $\LL^v$ is strictly elliptic because it was proved in section \ref{secao_eq_cur_media} that the infimum of the smallest eigenvalue of the matrix $(a_{ij}(\nabla v(x)))=(a_{ij}^v(x))$ is 1. 
Therefore, the operator
$$\funciones{T}{\cl^{1,\beta}(\overline{\W})}{\cl^{1,\beta}(\overline{\W})}{v}{u^v} $$
is well defined. 

In addition, the solvability of problem \eqref{ProblemaP} is equivalent to the existence of a fixed point of $T$. Indeed, since $Tv$ is the only solution of \eqref{ProblemaPv}, that is, 
\begin{equation*}
\left\{
\begin{split}
\sum_{i,j=1}^n \left(\left({W_v(x)}\right)^{2}\delta_{ij} - {\Di v \Dj v} \right)\Dij (Tv) &=nH(x) \left({W_v(x)}\right)^{3}\ \mbox{in}\ \W,\\
Tv&=\varphi \ \mbox{in}\ \partial\W,
\end{split}\right.
\end{equation*}
then the existence of a function $u\in\cl^{1,\beta}(\overline{\W})$ satisfying $Tu=u$ is exactly problem \eqref{ProblemaP}. {So, it gets evident the need for a theorem guaranteeing the existence of a fixed point of $T$. The following theorem is enough for our purpose.}

\medskip
\noindent\colorbox{shadecolor}{
\begin{minipage}{.975\textwidth}
\begin{teorema}[Leray-Schauder {\cite[T. 11.3 p. 280]{GT}}]
Let $\mathcal{B}$ be a Banach space and $T: \mathcal{B} \rightarrow \mathcal{B}$ a completely continuous operator\footnotemark. 
If there exists $M>0$ such that 
\begin{equation*}
\norm{x}_\mathcal{B} < M
\end{equation*}
for each $\tau\in [0,1]$ and $x\in \mathcal{B}$ satisfying $x=\tau T x$, then $T$ has a fixed point. 
\end{teorema}
\end{minipage}}
\footnotetext{We recall that a continuous mapping between two Banach spaces is called \textit{compact} or \textit{completely continuous} if the images of bounded sets are precompact.}

\medskip
In order to use the Leray-Schauder fixed point theorem we observe first that the family of solutions of the related problems \eqref{ProblemaPsigma} is not empty once $u = 0$ obviously satisfies $(P_0)$. Also $\mathcal{B}=\cl^{1,\beta}(\overline{\W})$ is a Banach space. 

In addition, since \eqref{ProblemaPv} holds for each $v\in\cl^{1,\beta}(\overline{\W})$, then for $\tau \in[0,1]$ it follows
\begin{equation}
\label{ProblemaPvtau}
\left\{
\begin{split}
\sum_{i,j=1}^n \left(\left({W_v(x)}\right)^2\delta_{ij} - {\Di v \Dj v} \right)\Dij (\tau Tv) &=\tau nH(x) \left({W_v(x)}\right)^{3}\ \mbox{in}\ \W,\\
\tau Tv&=\tau \varphi \ \mbox{in}\ \partial\W.
\end{split}\right.
\end{equation}
So, equation $\tau T u=u$ is equivalent to problem \eqref{ProblemaPsigma}, and the family of solutions of these problems is bounded in $\cl^1(\overline{\W})$ from the hypotheses.

The following global Hölder estimate for quasilinear operators of second order guarantees that we actually have an a priori bound in $\cl^{1,\beta}(\overline{\W})$ for some $\beta\in(0,1)$.
 
\medskip
\noindent\colorbox{shadecolor}{%
\begin{minipage}{.985\textwidth}
\begin{teorema}[Ladyzhenskaya-Ural'tseva {\cite[T. 13.7 p. 331]{GT}}]
Let $\Omega\subset \R^n$ be a bounded domain with $\partial\W$ of class $\cl^2$ and $\varphi\in\cl^2(\overline{\W})$. Assume that 
$u\in \cl^2( \overline{\Omega})$ satisfies
\begin{equation*}\left\{\begin{array}{l}\Q u =\ds\sum_{ij}a_{ij}(x,u,\nabla u)\Dij u + b(x,u,\nabla u)=0 \mbox{ in } \W,\\ 
\phantom{\Q u =\ds\sum_{ij}a_{ij}(x,u,\nabla u)\Dij u + b(x,u,\nabla u)} u = \varphi \mbox{ in } \partial\W,\end{array}\right.
\end{equation*}
where $\Q$ is an elliptic operator such that $a_{ij}\in \cl^1(\overline{\Omega} \times \R\times \R^n)$ and $b\in \cl^0(\overline{\Omega} \times \R\times \R^n)$.  
Then
\begin{equation*}
 [u_i]_{\alpha, \Omega} < C \ \ \forall i=1,\dots,n,
\end{equation*}
where $C=C\left(n,\norm{u}_{\cl^1(\overline{\W})},\norm{\varphi}_{\cl^2(\overline{\W})},\W,\lambda\right)$ and 
$\alpha=\alpha\left(n,\norm{u}_{\cl^1(\overline{\W})},\W,\lambda\right)>0$.
\end{teorema}
\end{minipage}}
\medskip

\noindent Indeed, the operator $\Q $ defined in \eqref{operador_Q} obviously satisfies the hypothesis of the previous theorem. Observe that this estimate does not depend on $u$ or $\tau$ since we already have an a priori estimate in $\cl^1(\overline{\W})$ for the family of solutions of \eqref{ProblemaPsigma}. {This is the constant $\beta$ that should be fixed at the beginning.}  

{
It still remains to prove that $T$ is continuous and compact. We recall first the following theorem for linear elliptic operators. 

\medskip
\noindent\colorbox{shadecolor}{%
\begin{minipage}{.985\textwidth}
\begin{teorema}[Global Schauder estimate {\cite[Th. 6.6 p. 98]{GT}}]
Let $\Omega\subset\mathbb{R}^n$ be a bounded domain with $\partial\W$ of class $\cl^{2,\alpha}$
and $\varphi\in \cl^{2,\alpha}(\overline{\Omega})$ for some $\alpha\in(0,1)$. Assume that $u\in \cl^{2,\alpha}(\overline{\Omega})$ satisfies 
\begin{align*}\label{probl}
\left\{\begin{array}{l}
\LL u=\ds\sum_{ij}a_{ij}(x)\Dij u+\ds\sum_{i}b_i(x)\Di u +c(x)u=f(x) \textrm{ in } \Omega,\\
\phantom{\LL u=\ds\sum_{ij}a_{ij}(x)\Dij u+\ds\sum_{i}b_i(x)\Di u +c(x)} u=\varphi\textrm{ in } \partial\Omega,
\end{array}\right.
\end{align*}
where $f$ and the coefficients of the strictly elliptic operator $\LL$ belong to $\cl^{\alpha}(\overline{\W})$.
Then
\begin{equation}\label{tezsches}
\norm{u}_{\cl^{2,\alpha}(\overline{\Omega})}\leq C\left(\|u\|_{\cl^0(\overline{\Omega})}+\|f\|_{\cl^{\alpha}(\overline{\Omega})}+\|\varphi\|_{\cl^{2,\alpha}(\overline{\Omega})}\right)
\end{equation}
for $C=C(\Omega, n, \alpha, \lambda, K )>0$ {where $K \geq \ds\max\left\{\norm{a_{ij}}_{\cl^{\alpha}(\overline{\Omega})},\norm{b_{i}}_{\cl^{\alpha}(\overline{\Omega})},\norm{c}_{\cl^{\alpha}(\overline{\Omega})}\right\}$}.
\end{teorema}
\end{minipage}}
\medskip

In view of the previous theorem 
$$\norm{T v}_{\cl^{2,\alpha\beta}(\overline{\Omega})}\leq C\left(\|T v\|_{\cl^0(\overline{\Omega})}+\norm{n H W_v^{3}}_{\cl^{\alpha\beta}(\overline{\Omega})}+\|\varphi\|_{\cl^{2,\alpha\beta}(\overline{\Omega})}\right)
$$
for every $v\in \cl^{1,\beta}(\overline{\W}) $. Besides, 
$$\|T v\|_{\cl^0(\overline{\Omega})}\leq \sup_{\partial\W}\mod{\varphi} + \tilde{C} \norm{nH W_v^{3}}_{\cl^{0}(\overline{\W})}$$
as a direct consequence of the following theorem for linear elliptic operators.

\medskip
\noindent\colorbox{shadecolor}{%
\begin{minipage}{0.985\textwidth}
\begin{teorema}[Height estimate {\cite[Th. 3.7 p. 36]{GT}}]
Let $\W\subset\R^n$ be a bounded domain. Assume that $u\in \cl^{2}(\W)\cap\cl^0(\overline{\Omega})$ satisfies 
$$\LL u=\ds\sum_{ij}a_{ij}(x)\Dij u+\ds\sum_{i}b_i(x)\Di u +c(x)u=f(x)$$ 
where $f$ and the coefficients $b_i$ of the strictly elliptic operator $\LL$ are bounded in $\overline{\W}$ and $c\leq 0$. 
Then
\begin{equation}\label{eq_est_apriori_linear}
\ds\sup_{\W}\mod{u}\leq \sup_{\partial\W}\mod{u}+C\sup_{\W}\mod{f}.
\end{equation}
where $C=C\left(\lambda,\diam(\W),\ds\max_{i}\sup_{\W}\mod{b_i}\right)$.
\end{teorema}
\end{minipage}}
\medskip

Observe that the constants $C$ and $\tilde{C}$ are independent of $v$. Hence, $T$ maps bounded sets in $\cl^{1,\beta}(\overline{\W})$ into bounded sets in $\cl^{2,\alpha\beta}(\overline{\W})\hookrightarrow \cl^{1,\beta}(\overline{\W})\hookrightarrow \cl^{\alpha}(\overline{\W})$. 

Let now $\left\{ T v_m\right\}$ be a sequence in $T\left(\cl^{1,\beta}(\overline{\W})\right)$. We affirm that there exists some subsequence that converges in $\cl^{1,\beta}(\overline{\W})$. 
Indeed, since for every $x$ and $y$ in $\overline{\W}$ we have
$$ \mod{T v_m(x)-T v_m(y)} \leq \norm{T v_m}_{\cl^{\alpha}(\overline{\W})} \norm{x-y}^{\alpha}, $$
so $\left\{T v_m\right\}$ is equicontinuous. This sequence is also uniformly bounded, then there exists a subsequence $\left\{Tv_{\xi(m)}\right\}$ that converges uniformly to a continuous function $w$ by the Arzela-Ascoli theorem. 

Besides, for every $1\leq i\leq n$,
$$ \mod{\Di T v_m(x)-\Di T v_m(y)}\leq \norm{Tv_m}_{\cl^{1,\beta}(\overline{\W})} \norm{x-y}^{\beta},$$
then $\left\{\Di Tv_m\right\}$ is also equicontinuous. By the Arzela-Ascoli theorem there exists a subsequence $\left\{\Di Tv_{\phi_i(m)}\right\}$ that converges uniformly to a continuous function $w_i$ since $\left\{\Di Tv_m\right\}$ is uniformly bounded. This implies that $\Di w$ exists, is continuous and $w_i=\Di w$.

Furthermore, for every $1\leq i, j\leq n$, 
$$ \mod{\Dij T v_m(x)-\Dij T v_m(y)}\leq \norm{Tv_m}_{\cl^{2,\alpha\beta}(\overline{\W})} \norm{x-y}^{\alpha\beta}.$$ Because of the $\cl^{2,\alpha\beta}$ estimate $\left\{\Dij T v_m\right\}$ is equicontinuous. Since it is also a bounded sequence, then there exists a subsequence $\left\{\Dij Tv_{\psi_{ij}(m)}\right\}$ that converges uniformly to $\Dij w$ (which necessarily exists and is continuous). 

This proves that $w\in\cl^{2}(\overline{\W})$ and that the subsequence $Tv_{\psi(\phi(\xi(m)))}$, with $\phi=\phi_1\circ\dots\circ\phi_n$ and $\psi=\psi_{nn}\circ\psi_{n,n-1}\circ\dots\circ\psi_{11}$, converges to $w$ in $\cl^{2}(\overline{\W})$ and also in $\cl^{1,\beta}(\overline{\W})$. 

In order to prove the continuity of $T$ let $\{v_m\}$ be a sequence converging to $v$ in $\cl^{1,\beta}(\overline{\W})$. Since $\{Tv_m\}$ is precompact in $\cl^2(\overline{\W})$, every subsequence has a convergent subsequence. Let $\left\{Tv_{\xi(m)}\right\}$ such a convergent subsequence with limit $w\in\cl^2(\overline{\W})$. Then
$$\left\{
\begin{split}
\sum_{i,j=1}^n \left(\left(W_{v_{\xi(m)}}\right)^2\delta_{ij} - {\Di v_{\xi(m)} \, \Dj v_{\xi(m)}}\right)\Dij \left(Tv_{\xi(m)}\right) &=nH(x) \left(W_{v_{\xi(m)}}\right) ^{3}\ \mbox{in}\ \W,\\
Tv_{\xi(m)}&=\varphi \ \mbox{in}\ \partial\W.
\end{split}\right.
$$
Passing to the limits and recalling the definition of $T$ one has $w=Tv$. 
Since every subsequence of $\{Tv_m\}$ has at least one subsequence which converges to $Tv$, then $\{Tv_m\}$ also converges to $Tv$.
}

So, it had been proved that $T$ has a fixed point $u\in\cl^{2,\alpha\beta}(\overline{\W})\hookrightarrow\cl^{1,\beta}(\overline{\W})$. To see that $u\in\cl^{2,\alpha}(\overline{\W})$ recall that $u$ is a solution of the linear problem $LP_u$ (see problem \eqref{ProblemaPv})
and use the following global regularity theorem for linear operators.

\medskip
\noindent\colorbox{shadecolor}{
\begin{minipage}{.97\textwidth}
\begin{teorema}[Global regularity {\cite[Th. 6.18 p. 111]{GT}}]
Let $\Omega\subset\mathbb{R}^n$ be a bounded domain with $\partial\W$ of class $\cl^{2,\alpha}$ and $\varphi\in \cl^{2,\alpha}(\overline{\Omega})$ for some $\alpha\in(0,1)$. Suppose that $u\in\cl^{2}(\W)$ satisfies 
\begin{align*}
\left\{\begin{array}{l}
\LL u=\ds\sum_{ij}a_{ij}(x)\Dij u+\ds\sum_{i}b_i\Di u +c(x)u=f \textrm{ in } \Omega,\\
\phantom{\LL u=\ds\sum_{ij}a_{ij}(x)\Dij u+\ds\sum_{i}b_i\Di u +c(x)} u=\varphi\textrm{ in } \partial\Omega,
\end{array}\right.
\end{align*}
where $f$ and the coefficients of the strictly elliptic operator $\LL$ belong to $\cl^{\alpha}(\overline{\W})$. Then $u\in\cl^{2,\alpha}(\overline{\W})$. 
\end{teorema}
\end{minipage}}
\medskip

Finally, notice that the solution is unique as a consequence of the comparison principle, Theorem \ref{PM_quasilineares}. 
\end{proof}
\begin{obs}
Applying the same argument to every fixed $\tau\in[0,1]$ we ensure the existence of a solution for each problem \eqref{ProblemaPsigma}. 
\end{obs}%


\section{Fulfillment of the existence program's requirements}\label{chapter_estimates}

The goal in this section is to obtain the a priori estimates for the family of solutions of the related problems \eqref{ProblemaPsigma} required by the existence program. 

In order to derive the a priori global gradient estimate the techniques introduced by Caffarelli-Nirenberg-Spruck \cite[p. 51]{CNSpruck} are used. 


\begin{teo}[A priori global gradient estimate {\cite[Th. 3.2.4 p. 94]{han2016nonlinear}}]\label{teo_Est_global_gradiente}
Let $\W\in M$ be a bounded domain with $\partial\W$ of class $\cl^2$. For $H\in\cl^{1}\left(\W\right)$, let $u\in\cl^3(\W)\cap\cl^1(\overline{\W})$ be a solution of \eqref{operador_minimo_1_coord}.
Then
\begin{equation}\label{est_global_final}
\sup_{\W}\norm{\nabla u(x)}\leq\left(\sqrt{3}+\sup\limits_{\partial\W}\norm{\nabla u}\right)\exp\left(2\sup\limits_{\W}\modulo{u}\left(1+8n\left(\norm{H}_1\right)\right)\right).
\end{equation}
\end{teo}
\begin{proof}
Let $w(x)=\norm{\nabla u(x)}e^{Au(x)}$ where $A\geq 1$. Suppose $w$ attains a maximum at $x_0\in\overline{\W}$. If $x_0\in\partial\W$, then
$$w(x)\leq w(x_0) =\norm{\nabla u(x_0)}e^{Au(x_0)}.$$
So,
\begin{equation}\label{est_global_1}
\sup_{\W}\norm{\nabla u(x)}\leq\sup_{\partial\W}\norm{\nabla u}e^{2A\sup\limits_{\W}\modulo{u}}. 
\end{equation}
Suppose now that $x_0\in\W$ and that $\nabla u(x_0)\neq 0$. {It can be assumed that} 
$e_1=\frac{\nabla u(x_0)}{\norm{\nabla u(x_0)}}$. Then, 
\begin{equation}\label{der_u_x0}
\Dk u(x_0)=\escalar{\Ek}{\nabla u(x_0)}
=\norm{\nabla u(x_0)}\delta_{k1}.
\end{equation}

Differentiating \eqref{operador_minimo_1_coord} with respect to $x_1$ we derive
\begin{equation}\label{der_MCE}
\begin{split}
\Dum \left(W^2\right)\Delta u+W^2 \ds\sum_{i=1}^n\Dumii u -2\ds\sum_{i,j=1}^n\Di u \Dumj u  \Hessij u -\sum_{i,j=1}^n \Di u \Dj u \Dumij u
\\
=n\Dum H W^3 + nH\Dum\left(W^3\right).
\end{split}
\end{equation}

The goal now is to estimate the third derivatives in \eqref{der_MCE}. Observe first that the function $\tilde{w}(x)=\ln w(x)=Au(x)+\ln \norm{\nabla u(x)}$ also attains a maximum at $x_0$. Therefore, for each $0\leq k \leq n$, 
\begin{equation}\label{derk}
\Dk \tilde{w}(x_0)= A\Dk u(x_0) + \dfrac{\Dk \left(\norm{\nabla u}^2\right)(x_0)}{2\norm{\nabla u(x_0)}^2}=0, 
\end{equation}
and
$$\Dkk \tilde{w}(x_0)=A\Dkk u(x_0) +\dfrac{1}{2}  \Dk\left({\norm{\nabla u}^{-2}}\right)(x_0)\Dk \left(\norm{\nabla u}^2\right)(x_0)+\dfrac{\Dkk \left(\norm{\nabla u}^2\right)(x_0)}{2\norm{\nabla u(x_0)}^2}\leq 0.$$
Since
$$\Dk\left({\norm{\nabla u}^{-2}}\right)=\Dk\left({\norm{\nabla u}^{2}}\right)^{-1}=-\left({\norm{\nabla u}^{2}}\right)^{-2} \Dk\left({\norm{\nabla u}^{2}}\right),$$
then
\begin{equation}\label{derkk}
A\Dkk u(x_0) - \dfrac{\left(\Dk \left(\norm{\nabla u}^2\right)(x_0)\right)^2}{2\norm{\nabla u(x_0)}^{4}}+\dfrac{\Dkk \left(\norm{\nabla u}^2\right)(x_0)}{2\norm{\nabla u(x_0)}^2}\leq 0.
\end{equation}

Once
\begin{equation}\label{Dknorma}
\Dk \left(\norm{\nabla u}^2\right)=\Dk \left(\sum_{i=1}^n(\Di u(x))^2\right)= 2 \ds\sum_{i=1}^n  \Di u  \Dki u,
\end{equation}
it follows from \eqref{der_u_x0} 
\begin{equation}\label{Dknorm}
\Dk \left(\norm{\nabla u}^2\right)(x_0)=2 \norm{\nabla u(x_0)}\Dumk u(x_0).
\end{equation}\vspace*{.1cm}

\noindent Substituting \eqref{der_u_x0} and \eqref{Dknorm} in \eqref{derk} one has
{
$$ A\norm{\nabla u (x_0)}\delta_{k1} + \dfrac{2 \norm{\nabla u(x_0)}\Dumk u(x_0)}{2\norm{\nabla u(x_0)}^2} =0,$$
thus,}%
\begin{equation}\label{derk_2}
\Dumk u(x_0)=-A\norm{\nabla u (x_0)}^2\delta_{k1}.
\end{equation}
\noindent Substituting also \eqref{derk_2} in \eqref{Dknorm} we obtain 
\begin{equation}\label{Dknorm_2}
\Dk \left(\norm{\nabla u}^2\right)(x_0)=-2A\norm{\nabla u (x_0)}^3\delta_{k1}.
\end{equation}

Besides, from \eqref{Dknorma} we get
\begin{align*}
 \Dkk \left(\norm{\nabla u}^2\right)(x) =&2\ds\sum_{i=1}^n \left( \Dkki u \Di u + (\Dki u)^2\right),
\end{align*}
and from \eqref{der_u_x0} we conclude
\begin{equation}\label{dkknorma_2}
\begin{split}
\Dkk \left(\norm{\nabla u}^2\right)(x_0)=2\norm{\nabla u(x_0)}\Dkkum u +2 \ds\sum_{i=1}^n (\Dki u(x_0))^2.
\end{split}
\end{equation}

Using expressions \eqref{Dknorm_2} and \eqref{dkknorma_2} in \eqref{derkk} we verify that 
\[
\begin{split}
A\Dkk u(x_0)-2A^2\norm{\nabla u (x_0)}^2\delta_{k1}+\ds \dfrac{\Dkkum u(x_0)}{\norm{\nabla u(x_0)}}
+\dfrac{\ds\sum_{i=1}^n\left(\Dki u(x_0)\right)^2}{\norm{\nabla u (x_0)}^2} &\leq 0.
\end{split}
\]
From \eqref{derk_2} we have for $k=1$ 
\[\begin{split}
-A^2\norm{\nabla u (x_0)}^2-2A^2\norm{\nabla u (x_0)}^2+\ds \dfrac{\Dumumum u(x_0)}{\norm{\nabla u(x_0)}}
+\dfrac{\ds\sum_{i=1}^n\left(-A\norm{\nabla u(x_0)}^2\right)^2\delta_{i1}}{\norm{\nabla u (x_0)}^2}& \leq 0,
\end{split}\]
then,
\begin{equation}\label{est_Dumumum}
\Dumumum u(x_0)\leq 2A^2\norm{\nabla u (x_0)}^3.
\end{equation}

\bigskip

\noindent If $k>1$, then
\begin{align*}
A\Dkk u(x_0)+\ds \dfrac{\Dkkum u(x_0)}{\norm{\nabla u(x_0)}}\leq-\dfrac{\ds\sum_{i=1}^n\left(\Dki u(x_0)\right)^2}{\norm{\nabla u (x_0)}^2} \leq 0,
\end{align*}
so,
\begin{equation}\label{est_Dkkum}
\Dkkum u(x_0)\leq -A\Dkk u(x_0)\norm{\nabla u (x_0)}, \ 1 < k\leq n .
\end{equation}

\bigskip

On the other hand, since \eqref{Dknorm_2} holds we deduce 
\begin{equation}\label{dW2}
\Dum\left(W^2\right)(x_0)=\Dum\left(\norm{\nabla u}^2\right)(x_0)=-2A\norm{\nabla u (x_0)}^3,
\end{equation}
and
\begin{equation}\label{dW3}
\Dum\left(W^3\right)(x_0)=\frac{3}{2}W_0\Dum \left(W^2\right)(x_0)=-3AW_0\norm{\nabla u (x_0)}^3.
\end{equation}
Using \eqref{der_u_x0}, \eqref{derk_2}, \eqref{est_Dumumum}, \eqref{est_Dkkum}, \eqref{dW2} and \eqref{dW3} in \eqref{der_MCE} it follows
\begin{align*}
&n\Dum H(x_0) W_0^3 -3nA H_0 W_0\norm{\nabla u(x_0)}^3\\[1em]
=&-2A\norm{\nabla u(x_0)}^3\Delta u(x_0)+W_0^2\ds\sum_{i=1}^n \Dumii u(x_0) \\
&+2A\norm{\nabla u(x_0)}^3\Dumum u(x_0)-\norm{\nabla u(x_0)}^2 \Dumumum u(x_0) \\[1em]
=&-2A\norm{\nabla u(x_0)}^3\ds\sum_{i>1}\Dii u(x_0)+W_0^2\ds\sum_{i>1} \Dumii u(x_0) +\Dumumum u(x_0) \\[1em]
\leq &-2A\norm{\nabla u(x_0)}^3\ds\sum_{i>1}\Dii u(x_0)-A\norm{\nabla u(x_0)}W_0^2\ds\sum_{i>1}\Dii u(x_0) +2A^2\norm{\nabla u(x_0)}^3 \\[1em]
=&-A\norm{\nabla u(x_0)}\left(1+3\norm{\nabla u(x_0)}^2\right)\ds\sum_{i>1}\Dii u(x_0) +2A^2\norm{\nabla u(x_0)}^3,
\end{align*}
where we have used the notation $H_0=H(x_0)$ and $W_0=\sqrt{1+\norm{\nabla u(x_0)}^2}$. 

In addition, since \eqref{der_u_x0} and \eqref{derk_2} holds, the mean curvature equation \eqref{operador_minimo_1_coord} at $x_0$ takes the form
\begin{align*}
 nH_0 W_0^3=W_0^2 \Delta u(x_0)-\norm{\nabla u(x_0)}^2\Dumum u(x_0)
=W_0^2\ds\sum_{i>1}\Dii u(x_0) -A \norm{\nabla u(x_0)}^2,
\end{align*}
so,
\begin{equation}\label{equ_curv_media_MxR_Delta_x0}
\ds\sum_{i>1}\Dii u(x_0) = nH_0W_0+\dfrac{A\norm{\nabla u(x_0)}^2}{W_0^2}.
\end{equation}
Therefore,
\begin{align*}
0\leq&-A\norm{\nabla u(x_0)}\left(1+3\norm{\nabla u(x_0)}^2\right)\left(nH_0W_0+\dfrac{A\norm{\nabla u(x_0)}^2}{W_0^2} \right)\\
&+2A^2\norm{\nabla u(x_0)}^3 + 3nA H_0 W_0\norm{\nabla u(x_0)}^3 - n\Dum HW_0^3 \\[1em]
=&-A n H_0 W_0\norm{\nabla u(x_0)}-n\Dum HW_0^3+\dfrac{A^2\norm{\nabla u(x_0)}^3}{W_0^2}\left(1-\norm{\nabla u(x_0)}^2\right).
\end{align*}
Then
\begin{align*}
\dfrac{A^2\norm{\nabla u(x_0)}^3}{W_0^2}\left(\norm{\nabla u(x_0)}^2-1\right)
\leq A n h_0 W_0\norm{\nabla u(x_0)} + n h_1 W_0^3,
\end{align*}
where $h_0=\sup\limits_{\W}\modulo{H}$ and $h_1=\sup\limits_{\W} \norm{\nabla H}$. Dividing by $A^2 W_0^3$ and noticing that $W_0^2> W_0 >\norm{\nabla u(x_0)}$ it follows 
$$ \dfrac{\norm{\nabla u(x_0)}^3}{W_0^5}\left(\norm{\nabla u(x_0)}^2-1\right)< \dfrac{n}{A} \norm{H}_1.  $$

Obviously we can suppose that $\norm{\nabla u(x_0)}>1$. Since 
$$W_0^3=\left(1+\norm{\nabla u(x_0)}^2\right)^{3/2}<\left(2\norm{\nabla u(x_0)}^2\right)^{3/2}
<4\norm{\nabla u(x_0)}^3, $$
we see that
$$\dfrac{1}{4}\dfrac{\norm{\nabla u(x_0)}^2-1}{W_0^2}<\dfrac{\norm{\nabla u(x_0)}^3}{W_0^3}\dfrac{\norm{\nabla u(x_0)}^2-1}{W_0^2}< \dfrac{n}{A} \norm{H}_1,  $$
that is,
$$\dfrac{\norm{\nabla u(x_0)}^2-1}{\norm{\nabla u(x_0)}^2+1}< \dfrac{4n}{A} \norm{H}_1.  $$
Choosing $A=1+8n\norm{H}_1$ it follows
$$\dfrac{\norm{\nabla u(x_0)}^2-1}{\norm{\nabla u(x_0)}^2+1}< \dfrac{1}{2},  $$
so, 
$$ \norm{\nabla u(x_0)}<\sqrt{3}.$$
As a consequence,
$$w(x)\leq w(x_0) =\norm{\nabla u(x_0)}e^{Au(x_0)}\leq\sqrt{3}e^{Au(x_0)},$$
thus 
\begin{equation}\label{est_global_grad_interior}
\sup_{\W}\norm{\nabla u(x)}\leq\sqrt{3}e^{2A\sup\limits_{\W}\modulo{u}}. 
\end{equation}

Putting together \eqref{est_global_1} and \eqref{est_global_grad_interior} we obtain
\begin{equation*}
\sup_{\W}\norm{\nabla u(x)}\leq\sqrt{3}e^{2A\scriptstyle\sup\limits_{\W}\modulo{u}} + \sup\limits_{\partial\W}\norm{\nabla u}e^{2A\scriptstyle\sup\limits_{\W}\modulo{u}},
\end{equation*}
which yields the desired estimate.
\end{proof}

The following theorem due to Serrin \cite[\S 9 p. 434]{Serrin} ensures the boundary gradient estimate. For the proof is used the classical idea of finding an upper and a lower barriers for $u$ on $\partial\W$ to get a control for $\nabla u$ along $\partial\W$. 

\begin{teo}[A priori boundary gradient estimate {\cite[Th. 3.2.2 p. 89]{han2016nonlinear}}]\label{teo_Est_gradiente_fronteira}
Let $\W\in \R^n$ be a bounded domain with $\partial\W$ of class $\cl^2$ and $\varphi\in\cl^2(\overline{\W})$. 
Let 
$H\in\cl^{1}\left(\overline{\W}\right)$ satisfying 
\begin{equation}\label{cond_Serrin}
(n-1)\Hc_{\partial\W}(y)\geq n \modulo{H(y)} \ \forall \ y\in\partial\W.
\end{equation}
If $u\in\cl^2(\W)\cap\cl^1(\overline{\W})$ is a solution of \eqref{ProblemaP}, then
\begin{equation}\label{EstGradFront}
\sup\limits_{\partial\W}\norm{\nabla u}\leq \norm{\varphi}_1 + \ds e^{C\left(1+ \norm{H}_1+ \norm{\varphi}_2\right)\left(1+\norm{\varphi}_1\right)^3\left(\norm{u}_0+\norm{\varphi}_0\right)}
\end{equation}
for some $C=C(n,\W)$.
\end{teo}
\begin{proof}
We set $d(x)=\dist(x,\partial\W)$ for $x\in\W$. Let $\tau>0$ be such that $d$ is of class $\cl^2$ over the set of points in $\W$ for which $d(x)\leq\tau$ (see \cite[L. 14.16 p. 355]{GT}, \cite[L. 3.1.8 p. 84]{han2016nonlinear} and \cite[L. 1 p. 420]{Serrin}).
Let $\psi\in\cl^2([0,\tau])$ be a non-negative function satisfying 
\begin{multicols}{3}
\begin{enumerate}
\item[P1.] $\psi'(t)\geq 1$,
\item[P2.] $\psi''(t) \leq 0$,
\item[P3.] $t\psi'(t)\leq 1$.
\end{enumerate}
\end{multicols}

For $a<\tau$ to be fixed latter on we consider the set 
$$\W_{a}=\left\{x\in M; d(x)<a \right\} .$$

We now define $w^{\pm}=\pm \psi\circ d +  \varphi$. Firstly, we estimate $\pm\M w^{\pm}$ in $\W_a$. 
Using the transformation formula \eqref{Eq_trans_varphi} one has
\begin{equation}\label{eq_barreira_superior}
\begin{split}
\pm \M w^{\pm}=&\psi'W_{\pm}^2\Delta d -\psi'\escalar{\Hess d \cdot  \nabla \varphi}{ \nabla \varphi} \\ 
& + \psi''W_{\pm}^2-\psi''\escalar{\nabla   d }{\pm \psi' \nabla d  +   \nabla \varphi}^2\\
&\pm W_{\pm}^2\Delta \varphi \mp   \escalar{\Hess\varphi\cdot (\pm\psi' \nabla d  +  \nabla \varphi)}{\pm\psi' \nabla d  +  \nabla \varphi},
\end{split}
\end{equation}
where 
$$
W_{\pm}=\sqrt{1+\norm{\nabla w^{\pm}}^2}=\sqrt{1+\norm{\pm \psi'\nabla  d  +   \nabla \varphi}^2}.
$$

Once $d$ is of class $\cl^2$ in $\W_a$ and $\psi'\geq 1$ we have 
\begin{equation}\label{BBB}
\psi'\modulo{\escalar{\Hess d \cdot \nabla \varphi}{ \nabla \varphi}} \leq \psi'^2 \norm{d}_2\norm{\varphi}_1^2.
\end{equation}

Since $\psi''<0$ and $\escalar{\nabla   d}{\pm\psi' \nabla d  +  \nabla \varphi}^2\leq \norm{\pm\psi' \nabla d + \nabla \varphi}^2$, then 
\begin{equation}\label{AAA}
\psi''W_{\pm}^2-\psi''\escalar{\nabla   d }{\pm\psi' \nabla d  +  \nabla \varphi}^2\leq \psi''.
\end{equation}

Also $\varphi$ is of class $\cl^2$ in $\W_a$ by hypothesis, so
\begin{align*}
&\modulo{\pm \Delta \varphi W_{\pm}^2\mp \escalar{\Hess \varphi \cdot (\pm\psi' \nabla d  +  \nabla \varphi)}{\pm\psi' \nabla d  +  \nabla \varphi} } \\ 
& \leq n\norm{\varphi}_2 W_{\pm}^2+ \norm{\varphi}_2\norm{\pm\psi' \nabla d  +  \nabla \varphi}^2\\
& \leq 2 n \norm{\varphi}_2 W_{\pm}^2.
\end{align*}
Notice now that
$$
\norm{\pm\psi' \nabla d  +  \nabla \varphi}^2=\left(\psi'^2+2 \psi'\escalar{\pm\nabla d}{\nabla \varphi}+\norm{\nabla \varphi}^2\right)
\leq \left(1+\norm{\varphi}_1\right)^2\psi'^2,
$$
hence
\begin{equation}\label{est_W2}
W_{\pm}^2 \leq  1 + \left(1+\norm{\varphi}_1\right)^2\psi'^2 \leq 2\left(1+\norm{\varphi}_1\right)^2\psi'^2.
\end{equation}
Therefore, 
\begin{equation}\label{CCC}
\begin{split}
\modulo{\pm \Delta \varphi W_{\pm}^2 \mp \escalar{\Hess \varphi \cdot (\pm\psi' \nabla d  +  \nabla \varphi)}{\pm\psi' \nabla d  +  \nabla \varphi} } \phantom{.}\\ 
\leq 4n\norm{\varphi}_2\left(1+\norm{\varphi}_1\right)^2\psi'^2.
\end{split}
\end{equation}

Substituting \eqref{AAA}, \eqref{BBB}, \eqref{CCC} in \eqref{eq_barreira_superior} it follows
\begin{equation}\label{est_Mwpm}
\pm\M w^{\pm}\leq \psi' W_{\pm}^2 \Delta d + \psi''+ c \psi'^2,
\end{equation}
where
\begin{equation}\label{constantec0}
c=\norm{d}_2\norm{\varphi}_1^2+4n\norm{\varphi}_2\left(1+\norm{\varphi}_1\right)^2.
\end{equation}
From \eqref{est_Mwpm} we obtain
\begin{equation}\label{eq_barreira_superior_Q_2}
\begin{array}{r}
\pm \Q_{ } w^{\pm} \leq  \psi'W_{\pm}^2\Delta d + \psi''+ c \psi'^2 + n  \modulo{H(x)}W_{\pm}^{3}.
\end{array}
\end{equation}

Let now $y=y(x)$ be the point of $\partial\W$ nearest to $x$. Then $x$ belongs to the segment $\{y+tN_y; 0\leq t \leq a\}$, where $N$ is the inner unit normal to $\partial\W$. Let us denote by $\Gamma_t$ the hypersurface parallel to $\partial\W$ at a distance $t$  contained in $\W_a$. We define $\Hc(t):=\Hc_{\Gamma_t}(y + t N_y)$ where $\Hc_{\Gamma_t}$ is the mean curvature of $\Gamma_t$ with respect to the normal that coincides with $\nabla d(x)=\frac{x-y}{\norm{x-y}}$ at $x$. 
Recall that $\Delta d(x)=-(n-1)\Hc_{\Gamma_{d(x)}}(x)$ (see \cite[\S 14.6 p 354]{GT}, \cite[\S 3.1 p. 80]{han2016nonlinear} and \cite[\S 3 p. 420]{Serrin}). In addition, 
$$\Hc_{\Gamma_t}(y+tN_y) \geq \Hc_{\partial\W}(y) $$
since $\Hc'(t)\geq (\Hc(t))^2 \geq 0$ in $[0,a)$ (see \cite[p. 485]{Serrin}). 
Using also the Serrin condition \eqref{cond_Serrin} we get
\begin{equation}\label{est_usar_hiperb}
\Delta d(x) \leq \Delta d(y) \leq -n \modulo{H(y)} \ \forall \ x\in\W_a
 \end{equation}

Substituting \eqref{est_usar_hiperb} in \eqref{eq_barreira_superior_Q_2} we obtain 
{\small
\begin{equation}\label{Qw_medio}
\begin{split}
\pm \Q_{ } w^{\pm} \leq &   n \psi'W_{\pm}^2( \modulo{H(x)} -\modulo{H(y)}) +n \modulo{H(x)}W_{\pm}^2 \left(W_{\pm}-\psi'\right) + \psi''+ c \psi'^2   .
\end{split}
\end{equation}}%

Besides,
$$
\modulo{H(x)}-\modulo{H(y)}\leq h_1(1+\norm\varphi_1)d(x),
$$
where $h_1=\sup\limits_{\W}\norm{\nabla  H }.$
Recalling also of \eqref{est_W2} one gets
\[n  \psi'W_{\pm}^2( \modulo{H(x)} -\modulo{H(y)} )\leq  2nh_1\left(1+\norm{\varphi}_1\right)^3 d(x)(\psi'(d(x)))^3.\]
Using the assumption P3 it follows
\begin{equation}\label{Termo2}
\begin{array}{c}
n  \psi'W_{\pm}^2( \modulo{H(x)} -\modulo{H(y)} )\leq 2 n h_1 \left(1+\norm{\varphi}_1\right)^3 \psi'^2.
\end{array}
\end{equation}

On the other hand, 
\begin{equation}\label{Wmenospsi_0}
W_{\pm}-\psi'\leq 1+\norm{\pm \psi'\nabla d +\nabla \varphi} -\psi' \leq 1+\norm{\varphi}_1.
\end{equation}
From \eqref{est_W2} and \eqref{Wmenospsi_0} we obtain 
\begin{equation}\label{Wmenospsi}
n \modulo{H(x)} \left(W_{\pm}-\psi'\right)W_{\pm}^2\leq 2 n h_0\left(1+\norm{\varphi}_1\right)^3\psi'^2,
\end{equation}
where $h_0=\sup\limits_{\W}\modulo{H}.$

Using \eqref{Termo2} and \eqref{Wmenospsi} in \eqref{Qw_medio} we get 
$$ \pm \Q_{ } w^{\pm} \leq \left(c+2n\norm{H}_{1}\left(1+\norm{\varphi}_1\right)^3\right)\psi'^2+\psi''.$$

Recalling the expression for $c$ given in \eqref{constantec0} and making some algebraic computations we infer that 
\begin{align*}
c+2n\norm{H}_{1}\left(1+\norm{\varphi}_1\right)^3 < C \left(1+\norm{\varphi}_2 + \norm{H}_{1}\right)\left(1+\norm{\varphi}_1\right)^3,
\end{align*}
where 
\begin{equation}\label{C_kappa}
C= 4n\left(1+\norm{d}_2+1/\tau\right).
\end{equation}
Choosing 
\begin{equation}\label{nu}
\nu= C \left(1+ \norm{H}_1+ \norm{\varphi}_2\right)\left(1+\norm{\varphi}_1\right)^3
\end{equation}
we define $\psi$ by 
$$
\psi(t)=\dfrac{1}{\nu}\log(1+kt). 
$$
So, 
\begin{equation}\label{dpsi}
\psi'(t)=\dfrac{k}{\nu(1+kt)}
\end{equation}
and 
\begin{equation}\label{ddpsi}
\psi''(t)=-\dfrac{k^2}{\nu(1+kt)^2},
\end{equation}
hence
$$
\pm \Q w^{\pm} < \nu\psi'^2+\psi''=0, \ \mbox{ in } \ \W_a.
$$
Besides
$$ t\psi'(t)=\dfrac{kt}{\nu(1+kt)}\leq \dfrac{1}{\nu}<1,$$
which is property P3. From \eqref{ddpsi} we see that property P2 is also satisfied. Another consequence of \eqref{ddpsi} is that $\psi'(t)>\psi'(a)$ for all $t\in[0,a]$, thus property P1 is ensured provided that
\begin{equation}\label{paraP1}
\psi'(a)=\dfrac{k}{\nu(1+ka)}=1.
\end{equation}

Furthermore, choosing 
\begin{equation}\label{esc_psia_curv_media}
\psi(a) = \dfrac{1}{\nu}\log(1+ka) =  \norm{u}_0+\norm{\varphi}_0,
\end{equation}
we have 
$$\pm w^{\pm}(x)=\psi(a)\pm \varphi(x)=\norm{u}_0+\norm{\varphi}_0 \pm \varphi(x) \geq \pm u(x) \  \forall \ x\in \partial\W_a\setminus\partial\W.$$ 
By combining \eqref{paraP1} and \eqref{esc_psia_curv_media} we see that 
\begin{equation}\label{cte_k}
k=\nu\ds e^{\nu(\norm{u}_0+\norm{\varphi}_0)}
\end{equation}
and, therefore, 
$$
a
=\dfrac{e^{\nu(\norm{u}_0+\norm{\varphi}_0)}-1}{\nu\ds e^{\nu(\norm{u}_0+\norm{\varphi}_0)}}.
$$
Note also that $a<\frac{1}{\nu}<\tau$ as required. 

Finally, if $x\in\partial\W$, then $w^{\pm}(x)=\pm \psi(0)+\varphi(x)=u(x)$. By the maximum principle we can conclude that $w^-\leq u \leq w^+ $ in $\W_a$, thus 
$$ -\psi\circ d  \leq u - \varphi \leq \psi\circ d \mbox{ in } \W_a.$$
Observe that 
$$ -\psi\circ d  = u - \varphi = \psi\circ d =0 \mbox{ in } \partial\W,$$
thus, for $y\in\partial\W$ and $0 \leq t \leq a$, we have 
$$-\psi(t) + \psi(0) \leq (u-\varphi) (y+t N_y) - (u-\varphi)(y) \leq \psi(t)-\psi(0).$$
Dividing by $t>0$ and passing to the limit as $t$ goes to zero we infer that
$$ -\psi'(0) \leq \escalar{\nabla u(y)-\nabla\varphi(y)}{N_y}\leq \psi'(0),$$
so,
$$ \mod{\escalar{\nabla u(y)-\nabla \varphi(y)}{N_y}}\leq \psi'(0).$$
Therefore,
\begin{equation}\label{est_grad_fronteira_1}
\modulo{\escalar{\nabla u(y)}{N_y}}\leq \modulo{\escalar{\nabla \varphi(y)}{N_y}} + \psi'(0).
\end{equation}
Since $u\equiv\varphi$ in $\partial\W$ 
$$ \nabla u(y)=(\nabla \varphi(y))^{T}+\escalar{\nabla u(y)}{N_y} N_y.$$
Using \eqref{est_grad_fronteira_1} we derive 
\begin{align*}
\norm{\nabla u(y)}^2=&\norm{(\nabla \varphi(y))^{T}}^2+{\escalar{\nabla u(y)}{N}}^2\\
\leq&\norm{(\nabla \varphi(y))^{T}}^2+ \left(\mod{\escalar{\nabla \varphi(y)}{N}} + \psi'(0)\right)^2\\
=&\norm{(\nabla \varphi(y))^{T}}^2+ \escalar{\nabla \varphi(y)}{N}^2 + 2\mod{\escalar{\nabla \varphi(y)}{N}} \psi'(0)+ (\psi'(0))^2\\
\leq&\norm{\nabla \varphi(y)}^2+ 2\norm{\nabla \varphi(y)} \psi'(0)+(\psi'(0))^2\\
=&\left(\norm{\nabla \varphi(y)} + \psi'(0)\right)^2.
\end{align*}
which yields the desired estimate.
\end{proof}

{The combination of Theorems \ref{teo_Est_global_gradiente} and \ref{teo_Est_gradiente_fronteira} with Theorem \ref{T_Exist_quaselineares} implies that the existence program for the Dirichlet problem \eqref{ProblemaP} is reduced to finding an a priori height estimate. In fact, the following theorem holds.}
\begin{teo}[{\cite[Th. 16.9 p. 407]{GT}}]\label{T_Exist_quaselineares_C0}
Let $\Omega\subset \R^n$ be a bounded domain with $\partial\W$ of class $\cl^{2,\alpha}$ for some $\alpha\in(0,1)$, $\varphi\in \cl^{2,\alpha}(\overline{\Omega})$ and {$H\in\cl^{1,\alpha}(\overline{\W})$}. Assume that the family of solutions of the related problems \eqref{ProblemaPsigma} is uniformly bounded. If 
\begin{equation}\label{StrongSerrinCondition_exist_c0}
(n-1)\Hc_{\partial\W}(y)\geq n \modulo{H\left(y\right)} \ \forall \ y\in\partial\W,
\end{equation}
then the Dirichlet problem \eqref{ProblemaP} has a unique solution in $\cl^{2,\alpha}(\overline{\W})$.
\end{teo}
\begin{proof} Let $M$ an a priori bound for the family of solutions of the related problems \eqref{ProblemaPsigma}. For $\tau\in[0,1]$ fixed, let $u$ a solution of \eqref{ProblemaPsigma}. 

Note now that Theorem \ref{teo_Est_global_gradiente} can be applied provided $u\in\cl^3({\W})$. Taking into account that $H\in\cl^{1,\alpha}(\overline{\W})$, that is a consequence of applying twice the following interior regularity theorem for linear operators. 

\medskip
\noindent\colorbox{shadecolor}{
\begin{minipage}{.98\textwidth}
\begin{teorema}[Interior regularity {\cite[Th. 6.17 p. 109]{GT}}]
Let $\Omega\subset\mathbb{R}^n$ a domain. Suppose that $u\in\cl^{2}(\W)$ satisfies 
$$\LL u=\ds\sum_{ij}a_{ij}(x)\Dij u+\ds\sum_{i}b_i(x)\Di u +c(x)u=f(x)$$ 
where $f$ and the coefficients of the elliptic operator $\LL$ belong to $\cl^{k,\alpha}({\W})$. Then $u\in\cl^{k+2,\alpha}({\W})$. 
\end{teorema}
\end{minipage}}
\medskip

\noindent{In fact, $u$ can be seen as a solution of the linear equation $\LL^u u= \tau nH(x)$ where $\LL^u$ is defined in \eqref{definitionLL}.} Therefore, 
\begin{align*}
\sup_{\W}\norm{\nabla u(x)}\leq&\left(\sqrt{3}+\sup\limits_{\partial\W}\norm{\nabla u}\right)\exp\left(2\sup\limits_{\W}\modulo{u}\left(1+8n\left(\tau\norm{H}_1\right)\right)\right)\\
\leq&\left(\sqrt{3}+\sup\limits_{\partial\W}\norm{\nabla u}\right)\exp\left(2M\left(1+8n\left(\norm{H}_1\right)\right)\right).
\end{align*}
But, on account of assumptions \eqref{StrongSerrinCondition_exist_c0} one has for any $y\in\partial\W$ 
$$(n-1)\Hc_{\partial\W}(y)\geq n \modulo{H(y)} \geq \tau n \modulo{H(y)}.$$
Then $u$ satisfies the estimate \eqref{EstGradFront} stated on Theorem \ref{teo_Est_gradiente_fronteira}. That is, there exists some constant $C=C(n,\W)$ such that
\begin{align*}
\sup\limits_{\partial\W}\norm{\nabla u}\leq & \norm{\tau\varphi}_1 + \ds e^{C\left(1+ \norm{H}_1+ \norm{\tau\varphi}_2\right)\left(1+\norm{\tau\varphi}_1\right)^3\left(\norm{u}_0+\norm{\tau\varphi}_0\right)}\\
\leq & \norm{ \varphi}_1 + \ds e^{C\left(1+ \norm{H}_1+ \norm{ \varphi}_2\right)\left(1+\norm{ \varphi}_1\right)^3\left(M+\norm{ \varphi}_0\right)}.
\end{align*}


Thus, the family of solutions of the related problems \eqref{ProblemaPsigma} is bounded in $\cl^1(\overline{\W})$ independently of $\tau$. Theorem \ref{T_Exist_quaselineares} ensures the existence of a unique solution $u\in\cl^{2,\alpha}(\overline{\W})$ for our problem \eqref{ProblemaP}. 
\end{proof}

The following theorem guarantees an a priori height estimate if the function $H$ satisfies a further hypothesis in addition to the Serrin condition. 
\begin{teo}[A priori height estimate {\cite[p. 484]{Serrin}}]\label{teo_Est_altura}
Let $\W\in \R^n$ be a bounded domain with $\partial\W$ of class $\cl^2$. Let $H\in\cl^{1}(\overline{\W})$ satisfying 
\begin{equation}\label{cond_H_Ricci_sup}
\norm{\nabla H(x)} \leq \dfrac{n}{n-1}\left(H(x)\right)^2 \ \forall \ x\in\W
\end{equation}
and
\begin{equation}\label{cond_Serrin_hightest_teo}
(n-1)\Hc_{\partial\W}(y)\geq n \modulo{H(y)} \ \forall \ y\in\partial\W.
\end{equation}
If $u\in\cl^2(\W)\cap\cl^0(\overline{\W})$ is a solution of the mean curvature equation \eqref{operador_minimo_1_coord} in $\W$, then 
\begin{equation}\label{Est_Altura}
\sup\limits_{\W}\modulo{u}\leq \sup\limits_{\partial\W} \modulo{u} +\dfrac{e^{\mu\delta}-1}{\mu},
\end{equation}
where $\mu>n\sup\limits_{\overline{\W}}\modulo{H}$ and $\delta=\diam(\W)$.
\end{teo}
\begin{proof}
Let $d(x)=\dist(x,\partial\W)$ for $x\in\W$. Let $\W_0$ be the biggest open subset of $\W$ having the unique nearest point property,
{then $d\in\cl^2(\W_0)$ (see \cite[p. 409, Lemmas 14.16 and 14.17 p. 355]{GT}, \cite[p. 481, \S 3 p. 420]{Serrin}).}

We now define $w=\phi\circ d + \sup\limits_{\partial\W}\modulo{u}$ over $\W$, where 
$$\phi(t)=\dfrac{\ds e^{\mu\delta}}{\mu}\left(1-e^{-\mu t}\right).$$ 

If we prove that $\mod{u}\leq w$ in $\overline{\W}$ we obtain the desired estimate. By the sake of contradiction we suppose first that the function $v=u-w$ attains a maximum $m>0$ at $x_0\in{\W}$ (note that $u\leq w$ in $\partial\W$). 

Let $y_0\in\partial\W$ be such that $d(x_0)=\dist(x_0,y_0)=t_0$ and $\gamma$ the straight line segment joining $x_0$ to $y_0$. Restricting $u$ and $w$ to $\gamma$ we see that $v'(t_0)=0$. 
Hence, 
$u'(t_0)=w'(t_0)=\phi'(t_0)>0$ which implies that $\nabla u(x_0)\neq 0$. Therefore, 
$\Gamma_0=\left\{x\in\W;u(x)=u(x_0)\right\}$ is of class $\cl^2$ near $x_0$. Then, there exists a small ball $B_{\epsilon}(z_0)$ tangent to $\Gamma_0$ in $x_0$ such that 
\begin{equation}\label{eq_bola_1}
u > u(x_0) \mbox{ in } \overline{B_{\epsilon}(z_0)}\setminus\{x_0\}. 
\end{equation}
We note that
$$\dist(z_0,y_0)\leq \dist(z_0,x_0)+\dist(x_0,y_0)=\epsilon + d(x_0).$$
Hence, for $\tilde{z}$ lying in the intersection of $\partial B_{\epsilon}(z_0)$ with the straight line segment joining $z_0$ to $y_0$, we have
$$d(\tilde{z})\leq \dist(\tilde{z},y_0)=\dist(z_0,y_0)-\epsilon \leq d(x_0)+\epsilon -\epsilon = d(x_0).$$
Thus, $w(\tilde{z})\leq w(x_0)$ since $\phi$ is increasing. Consequently,
$$ u(\tilde{z})-w(x_0) \leq u(\tilde{z})-w(\tilde{z}) \leq u(x_0)-w(x_0)$$
and $u(\tilde{z})\leq u(x_0)$. By \eqref{eq_bola_1} one has that $\tilde{z}=x_0$, so $z_0$ belongs to $\gamma$ and $\gamma$ is orthogonal to $\Gamma_0$.  
This ensures that $x_0\in\W_0$ because if there exists $y_1\neq y_0$ satisfying $d(x_0)=\dist(x_0,y_1)$, then 
the straight line joining $y_1$ and $x_0$ is also orthogonal to $\Gamma_0$, which is a contradiction.

However, let us show that this is also impossible. Using the transformation formula \eqref{calc_Q_w} one has
\begin{equation}\label{Mw_est_altura_0}
\M w =  \phi'(1+\phi'^2) \Delta d + {\phi''} \ \mbox{ in } \W_0.
\end{equation}
We first estimate $\Delta d$ in $\W_0$. For $x\in\W_0$, let $y=y(x)$ in $\partial\W$ be the nearest point to $x$, so $x$ belongs to the segment $\{y + t N_y; t>0\}$, where $N$ is the inner normal to $\partial\W$. Note that $y$ is now fixed. {Let us denote by $\{\Gamma_t\}$ the hypersurface parallel to some portion of $\partial\W$ containing $y$ at distance $t$. So $x$ belongs to $\Gamma_{d(x)}$.

Let 
$$h(t)=\frac{n}{n-1}H\left(y + t N_y\right).$$ 
Therefore
$$ h'(t)=\dfrac{n}{n-1}\escalar{\nabla H(y + t N_y)}{N_y}. $$
Taking into account the additional hypothesis \eqref{cond_H_Ricci_sup} we see that
$$\modulo{h'(t)}\leq \dfrac{n}{n-1}\norm{\nabla H(y + t N_y)} \leq   (h(t))^2 ,$$
hence
$$\modulo{h'(t)}-  (h(t))^2 \leq 0.$$
Recalling again that $\Hc'(t)\geq (\Hc(t))^2 $ (see \cite[p. 485]{Serrin}) it follows 
\begin{equation}\label{desig_sem_ricc}
\Hc'(t) \geq  \left(\Hc(t)\right)^2 + \modulo{h'(t)} - (h(t))^2 .
\end{equation}
Then,
\begin{equation}\label{desig_sem_ricc_v}
(\Hc(t) - h(t))'\geq \left(\Hc(t)+h(t)\right)\left(\Hc(t)-h(t)\right)
\end{equation}
and
\begin{equation}\label{desig_sem_ricc_g}
(\Hc(t) + h(t))'\geq \left(\Hc(t)-h(t)\right)\left(\Hc(t)+h(t)\right).
\end{equation}
Let us define $v(t)=\Hc(t)-h(t)$ and $g(t)=\Hc(t)+h(t)$. 
From \eqref{desig_sem_ricc_v} one has 
$$v'(t) \geq  g(t) v(t)$$
Multiplying this inequality by $\ds e^{\int_0^t g(s)ds}$, it results 
$$\left(\dfrac{v(t)}{\ds e^{\int_0^t g(s)ds}}\right)'\geq 0,$$
so 
$$\ds\dfrac{v(t)}{\ds e^{\int_0^t g(s)ds}}\geq v(0)=\Hc(0)-h(0).$$
From the Serrin condition \eqref{cond_Serrin_hightest_teo} it follows that
$$\modulo{h(0)} = \dfrac{n}{n-1} \modulo{H\left(y\right)}\leq \Hc_{\partial\W}(y) =\Hc(0) .$$
thus, $v(t)\geq 0$ and $\Hc(t)\geq h(t).$

Using \eqref{desig_sem_ricc_g} we obtain in a similar way that $\Hc(t)\geq -h(t).$
Therefore, 
$$
\Hc(t)\geq\modulo{h(t)},
$$
that is, 
$$ n \modulo{H(y+tN_y)}\leq  (n-1)\Hc_{\Gamma_t}(y+tN_y).$$
Consequently,
\begin{equation}\label{forRemark} 
n \modulo{H(x)}\leq  (n-1)\Hc_{\Gamma_{d(x)}}(x).
\end{equation}
This proves that
$$\Delta d(x)\leq-n\modulo{H\left(x\right)} \ \forall \ x\in\W_0.$$
}
Using this estimate in \eqref{Mw_est_altura_0} we have in $\W_0$
$$\M w \leq  -n\modulo{H(x)} {\phi'}{(1+\phi'^2)}  + {\phi''}.$$
Also
$$
\phi''(t)=-\mu \ds e^{\mu(\delta-t)}=-\mu\phi'(t)<-n\modulo{H(x)}\phi'(t) 
$$
and $\phi'\geq 1$, so
%
\begin{equation}\label{est_Mw_est_alt}
\M w \leq -n\modulo{H\left(x\right)} {\phi'(2+\phi'^2)}< -n\modulo{H\left(x\right)}{\left(1+\phi'^2\right)^{3/2}}.
\end{equation}
From \eqref{est_Mw_est_alt} we conclude that 
\begin{align*} 
\Q (w+m) =\Q w = & \M w - nH\left(x\right) {\left(1+\phi'^2\right)^{3/2}}\leq 0=\Q u.
\end{align*}
Moreover, $u \leq w + m$ and $u(x_0)=w(x_0)+m$. By the maximum principle $u\equiv w+m$ in $\W_0$ which is a contradiction since $u<w+m$ in $\partial\W$. This proves that $u\leq w$ in $\overline{\W}$. 

Applying the same argument to the function $-u$ we also obtain that $-u\leq w$ in $\overline{\W}$. 
\end{proof}
\begin{obs}
The proof shows that if there exists a function $H$ satisfying the hypothesis \eqref{cond_H_Ricci_sup} in addition to the Serrin condition \eqref{cond_Serrin_hightest_teo}, then any hypersurface that is parallel to some portion of $\partial\W$ ``inherit'' the Serrin condition (see \eqref{forRemark}).  This geometric implication is the key to obtain the height estimate for solutions of the mean curvature equation in terms of its boundary values. 
\end{obs}
\begin{obs} 
The proof shows that condition \eqref{cond_H_Ricci_sup} only needs to be valid in $\W_0$. 
\end{obs}

We are able to prove the following theorem from Serrin. 
\begin{teo}[Serrin {\cite[p. 484]{Serrin}}]\label{T_exist_Ricci}
Let $\Omega \subset \R^n$ be a bounded domain with $\partial\W$ of class $\cl^{2,\alpha}$ for some $\alpha\in(0,1)$. 
Let $H\in\cl^{1,\alpha}(\overline{\W})$ satisfying 
\begin{equation}\label{cond_H_Ricci_exist}
\norm{\nabla H(x)} \leq \dfrac{n}{n-1}\left(H(x)\right)^2 \ \forall \ x\in\W.
\end{equation}
and
\begin{equation}\label{StrongSerrinCondition_exist}
(n-1)\Hc_{\partial\W}(y)\geq n \modulo{H\left(y\right)} \ \forall \ y\in\partial\W.
\end{equation}
Then for every $\varphi\in\cl^{2,\alpha}(\overline{\W})$ there exists a unique solution $u\in\cl^{2,\alpha}(\overline{\W})$ of the Dirichlet problem \eqref{ProblemaP}.
\end{teo}
\begin{proof}
By Theorem \ref{T_Exist_quaselineares_C0} it only remains to prove that the family of solutions of the related problems \eqref{ProblemaPsigma} is uniformly bounded. Let $u$ be a solution of problem \eqref{ProblemaPsigma} for arbitrary $\tau\in[0,1]$ and let $w=\phi\circ d + \sup\limits_{\partial\W}\modulo{\varphi}$ as in the proof of Theorem \ref{teo_Est_altura}. Analogously as in the proof of that theorem, if the function $u-w$ attains a positive maximum $m$ at $x_0\in \overline{\W}$, then $x_0$ would be in $\W_0$ (the biggest open subset of $\W$ having the unique nearest point property). 
But for $x\in\W_0$ we have 
\begin{align*}
\Q_{\tau}(w + m )=  \Q_{\tau}(w)
\leq \M w + \tau n\modulo{H(x)}(1+\phi'^2)^{3/2}\leq 0
\end{align*}
once \eqref{est_Mw_est_alt} holds and $\tau\in[0,1]$. Proceeding as in the proof of Theorem \ref{teo_Est_altura}, we get that $u\leq w$ also in $\W_0$. Analogously it follows that $-u \leq w$ in $\overline{\W}$. Hence, $u$ satisfies the estimate \eqref{Est_Altura}, that is,
$$\sup\limits_{\W}\modulo{u}\leq \sup\limits_{\partial\W} \modulo{\tau\varphi} +\dfrac{e^{\mu\delta}-1}{\mu}
\leq \sup\limits_{\partial\W} \modulo{\varphi} +\dfrac{e^{\mu\delta}-1}{\mu},$$
where $\mu>n\sup\limits_{\overline{\W}}\modulo{H}$ and $\delta=\diam(\W)$.
\end{proof}

\begin{obs}
Observe that Theorem \ref{T_exist_Ricci} is not exactly the existence part in Theorem \ref{T_Serrin_Ricci} stated in the introduction. In order to reduce the differentiability assumptions on the domain and the boundary data, the original problem can be approximated by new problems having the $\cl^{2,\alpha}$ differentiability requirements. For instance, $\partial\W$ can be approximated by $\cl^{2,\alpha}$ hypersurfaces. However, it can only be guaranteed that the mean curvature of any approximating surface $\Sigma$ satisfies $(n-1)\Hc_{\Sigma}(x)\geq n\mod{H(x)}-\varepsilon$ for every $x\in\Sigma$, where $\varepsilon$ is a positive constante. Hence, a boundary gradient estimate sharper than that obtained in Theorem \ref{teo_Est_gradiente_fronteira} is needed. Furthermore, it is also required an interior gradient estimate which yields a compactness result also used in this argument. We refer the work of Serrin \cite[\S 14 p. 451]{Serrin} for further studies. 
\end{obs}

\begin{obs}
Although the additional hypothesis \eqref{cond_H_Ricci_exist} is required in order to obtain an a priori height estimate, the fundamental role in Theorem \ref{T_exist_Ricci} is played by the Serrin condition \eqref{StrongSerrinCondition_exist}. 
If fact, if the function $H$ satisfies the integral condition 
$$ \norm{H}_{L^n(\W)}<n\left(\ds\int_{\R^n}\left(1+\norm{p}^2\right)^{-\frac{n+2}{2}}dp\right)^{\frac{1}{p}}$$
instead of \eqref{cond_H_Ricci_exist}, the height estimate is guaranteed and the same conclusion of Theorem \ref{T_exist_Ricci} holds as a consequence of Theorem \ref{T_Exist_quaselineares_C0} (see \cite[Ths. 3.2.1 p. 87 and 3.4.1 p. 105]{han2016nonlinear} and \cite[Th. 16.10 p. 408]{GT}). 
\end{obs}

\section{Sharpness of the Serrin condition}\label{cap_NaoExis}

The goal in this section is to prove that the Serrin condition,
\begin{equation}\label{SerrinCondition_naoexist}
(n-1)\Hc(y)\geq n\mod{H(y)} \ \forall \ y\in\partial\W, 
\end{equation}
is actually sharp for the solvability of the Dirichlet problem \eqref{ProblemaP}.  
That is, if \eqref{SerrinCondition_naoexist} fails, then there exists boundary values for which problem \eqref{ProblemaP} has no possible solution.

The next lemma is an important peace. In this lemma is established a height a priori estimate for solutions of equation 
\eqref{operador_minimo_1_coord} in $\W$ in those points of $\partial\W$ on which the Serrin condition \eqref{SerrinCondition_naoexist} fails.

\begin{lema}[{\cite[L. 3.4.4 p. 109]{han2016nonlinear}}]\label{M_nao_exist_MxR_estimativa_Hxz}
Let $\W\subset M$ be a bounded domain whose boundary is of class $\cl^2$. Let $H\in\cl^0(\overline{\W})$ be a non-negative function and $u\in\cl^2(\W)\cap\cl^0(\overline{\W})$ satisfying \eqref{operador_minimo_1_coord}. 
Assume that there exists $y_0\in\partial\W$ such that
\begin{equation}\label{cond_Serrin_negac_M_pos}
(n-1)\Hc_{\partial\W}(y_0)<nH(y_0).
\end{equation}
Then for each $\varepsilon>0$ there exists $a>0$ depending only on $\varepsilon$, $\Hc_{\partial\W}(y_0)$, the geometry of $\W$ and the modulus of continuity of $H$ in $y_0$,  such that
\begin{equation}\label{est_nao_exist_Hxz}
u(y_0) < \ds\sup_{\partial\W\setminus B_a(y_0)}~u+\varepsilon. 
\end{equation}
\end{lema}

\begin{proof}
The proof is done in two steps. Firstly, it will be find an estimate for $u(y_0)$
depending on $\sup\limits_{\partial B_a(y_0)\cap \W}u$ for some $a$ that does not depend on $u$. Secondly, an upper bound for $\ds \sup_{\partial B_a(y_0)\cap \W}u$ in terms of $\sup\limits_{\partial\W\setminus B_a(y_0)} u $ is stated. 

\bigskip
\noindent\textbf{Step 1.} 
First of all note that from \eqref{cond_Serrin_negac_M_pos} there exists $\nu>0$ such that
\begin{equation}\label{M_condcomigual}
(n-1)\Hc_{\partial\W}(y_0) < n H(y_0)-4\nu.
\end{equation}
Let $R_1>0$ be such that $\partial B_{R_1}(y_0)\cap\W$ is connected and
\begin{equation}\label{M_cond_H}
\modulo{H(x)-H(y_0)}<\dfrac{\nu}{n}, \ \forall \ x\in B_{R_1}(y_0)\cap\W.
\end{equation}

Let $S$ be a quadric hypersurface inside $\W$, 
tangent to $\partial\W$ at $y_0$ and whose mean curvature calculated with respect to the normal field $N$ which coincides with the inner normal to $\partial\W$ at $y_0$ 
satisfies
\begin{equation}\label{M_curv_S_Gamma_ponto}
\Hc_{S}(y_0)<\Hc_{\partial\W}(y_0)+\dfrac{\nu}{(n-1)}.
\end{equation}

{
Let $d(x)=\dist(x,S)$ for $x\in\W$. It is known that $d$ is of class $\cl^2$ over the strip
$$\Sigma_{\tau}=\{x+tN_x; x \in S , \ t\in[0,\tau) \}$$ 
for some $\tau>0$ (see \cite[L. 3.1.8 p. 84]{han2016nonlinear} and \cite[L. 1 p. 420]{Serrin}). 
Besides, for each $t\in[0,\tau)$ fixed, 
$$S_{t}=\{x+tN_x; x \in S \}$$ 
is parallel to $S$. 
Let $0<R_2<\min\{\tau,R_1\}$ be such that
\begin{equation}\label{est_Laplaciano_d}
\modulo{\Delta d(x)-\Delta d(y_0)}<\nu  \ \ \forall \ x\in B_{R_2}(y_0)\cap \Sigma_{\tau}.
\end{equation}%
}

Let us fix $a<R_2$ to be made precise later. For $0<\epsilon<a$ let
$$\W_{\epsilon}=\{x \in B_a(y_0)\cap\Sigma_{\tau}; d(x)>\epsilon\}.$$

Let $\phi\in\cl^2(\epsilon, a)$ satisfying
{\small%
\begin{multicols}{4}
\begin{enumerate}
\item[P1.] $\phi(a)=0$,
\item[P2.] $\phi'\leq0$,
\item[P3.] $\phi''\geq0$,
\item[P4.] $\phi'(\epsilon)=-\infty$.
\end{enumerate}
\end{multicols}}%
\noindent It is also required that 
\begin{equation}\label{assumption_phi}
\nu (\phi'(t))^3 + \phi''(t)=0, \ t\in (\epsilon,a). 
\end{equation}

Let us define $v = \ds\sup_{\partial B_a(y_0)\cap\W} u   + \phi\circ d$. So, $v\geq u$ in $\partial \W_{\epsilon}\setminus S_{\epsilon}$. If $u\leq v$ in $S_{\epsilon}$, then an estimate for $u(y_0+\epsilon N_{y_0})$ is obtained. Observe now that if $N_{\epsilon}$ is the normal to $S_{\epsilon}$ inwards $\W_{\epsilon}$ and $x\in S_{\epsilon}\cap B_a(y_0)$, then
\begin{align*}
\parcial{v}{N_{\epsilon}}(x)&=\escalar{\nabla v(x)}{N_{\epsilon}(x)}=\escalar{\phi'(d(x))\nabla d(x)}{\nabla d(x)}=\phi'(\epsilon)=-\infty.
\end{align*}
So, Proposition \ref{M_prop_gen_JS} can be used if $\Q u \geq \Q v$ in $\W_\epsilon$. This will be proved in the sequel.

For $x\in \W_{\epsilon}$ the transformation formula \eqref{calc_Q_w} yields
$$\Q v =  \phi'(1+\phi'^2) \Delta d + \phi''-n H(x) (1+\phi'^2)^{3/2}.$$
The assumptions on $\phi$ immediately gives
$$(1+\phi'^2)^{3/2}
=(1+\phi'^2)^{1/2}(1+\phi'^2)
>(\phi'^2)^{1/2}(1+\phi'^2)=\modulo{\phi'}(1+\phi'^2)=-\phi'(1+\phi'^2).$$
Since $H\geq 0$, then 
$$-nH(x)(1+\phi'^2)^{3/2} < nH(x){\phi'}{(1+\phi'^2)} .$$
Therefore,
\begin{equation}\label{M_exp_Qv_3}
\Q v < {\phi'}{(1+\phi'^2)} \left(\Delta d(x) + n H(x)  \right) + {\phi''}.
\end{equation}
Furthermore,
\begin{align*}
\Delta d(x) + n H(x) = & \Delta d(x) - \Delta d(y_0) -(n-1)\Hc_S(y_0) + n H(x)\\
												> & -\nu -(n-1)\Hc_S(y_0) + n H(x)\tag{a}\\
										    > & -2\nu -(n-1)\Hc_{\partial\W}(y_0) + n H(x)\tag{b}\\
										    > &  2\nu - n H(y_0) + n H(x) \tag{c}\\
										    > & \nu \tag{d},
\end{align*}
where (a) follows directly from \eqref{est_Laplaciano_d}, (b) from \eqref{M_curv_S_Gamma_ponto}, (c) from \eqref{M_condcomigual} and (d) from \eqref{M_cond_H}. Using this estimate in \eqref{M_exp_Qv_3} it follows
\begin{align*}
\Q v < \phi'(1+\phi'^2) \nu + \phi'' < \phi'^3 \nu + \phi''. 
\end{align*}
Assumption \eqref{assumption_phi} yields $\Q v <0$ in $\W_{\epsilon}$. From Proposition \ref{M_prop_gen_JS} it is deduced that
$$ u \leq v = \ds\sup_{\partial B_a(y_0)\cap\W} u + \phi(\epsilon) \ \ \mbox{in}\ \ S_{\epsilon}\cap B_a(y_0). $$

Let us now define $\phi$ explicitly by (see also \cite[\S 14.4]{GT})
\begin{equation}\label{M_exp_phi}
\phi(t)=\sqrt{\dfrac{2}{\nu}}\left((a-\epsilon)^{1/2}-(t-\epsilon)^{1/2}\right).
\end{equation}
Observe that $\phi$ satisfies P1--P4 and that $\phi'^3 \nu + \phi''=0$ in $(\epsilon,a)$.
{
Indeed,
$$ \phi'(t)=-\dfrac{1}{2}\sqrt{\dfrac{2}{\nu}}(t-\epsilon)^{-1/2}=-\dfrac{1}{2}\left(\dfrac{2}{\nu(t-\epsilon)}\right)^{1/2}$$
and
$$ \phi''(t)=\dfrac{1}{4}\sqrt{\dfrac{2}{\nu}}(t-\epsilon)^{-3/2}=\dfrac{\nu}{8}\left(\dfrac{2}{\nu(t-\epsilon)}\right)^{3/2}=-\nu \phi'(t)^3.$$
}%
Therefore, 
$$ u(y_0 + \epsilon N_{y_0}) \leq \ds\sup_{\partial B_a(y_0)\cap\W} u + \sqrt{\dfrac{2}{\nu}}\left((a-\epsilon)^{1/2}\right).$$
Since this estimate holds for each $0<\epsilon<a$, we can pass to the limit as $\epsilon$ goes to zero, so
\begin{equation}\label{est_u0_1}
u(y_0) \leq  \ds\sup_{\partial B_a(y_0)\cap\W} u  + \sqrt{\dfrac{2a}{\nu}}.
\end{equation}


\bigskip

\noindent\textbf{Step 2.} 
Let $\delta=\diam(\W)$ and $\psi\in\cl^2(a,\delta)$ satisfying
{\small%
\begin{multicols}{4}
\begin{enumerate}
\item[P5.] $\psi(\delta)=0$,
\item[P6.] $\psi'\leq0$,
\item[P7.] $\psi''\geq0$,
\item[P8.] $\psi'(a)=-\infty$.
\end{enumerate}
\end{multicols}}%
\noindent It is also needed that 
\begin{equation}\label{assumptionpsi}
(n-1)\frac{(\psi'(t))^3}{t}+\psi''(t)\leq 0, \ t\in(a,\delta). 
\end{equation}

Let $w=\ds\sup_{\partial\W\setminus B_a(y_0)} u + \psi\circ\rho$, where $\rho(x)=\dist(x,y_0)$. Remind that $\rho\in\cl^2(\R^n\setminus\{y_0\})$, so $w\in\cl^2(\W\setminus B_a(y_0))$. 
The idea is to use Proposition \ref{M_prop_gen_JS} again. 
Note that $w\geq u$ in $\partial\W\setminus B_a(y_0)$.
Also, if $N_a$ is the normal to $\partial B_a(y_0)\cap\W$ inwards $\W\setminus B_a(y_0)$ and $x\in\partial B_a(y_0)\cap\W$, then
\begin{align*}
\parcial{w}{N_a}(x)&=\escalar{\nabla w(x)}{N_a(x)}=\escalar{\psi'(\rho(x))\nabla \rho(x)}{\nabla \rho(x)}=\psi'(a)=-\infty.
\end{align*}

On the other hand, the transformation formula \eqref{calc_Q_w} gives
$$ \Q w =  {\psi'}{(1+\psi'^2)} \Delta \rho + {\psi''}-n H(x) {(1+\psi'^2)^{3/2}}.$$
Since $H\geq 0$ and $\Delta \rho(x) = \dfrac{n-1}{\rho(x)}$ 
it follows
\begin{align*}
 \Q w \leq  \dfrac{n-1}{\rho} \psi'(1+\psi'^2)  + \psi''<\dfrac{n-1}{\rho} \psi'^3  + \psi''.
\end{align*}
Assumption \eqref{assumptionpsi} yields $\Q w <0$ in $\W\setminus B_a(y_0)$. 

From Proposition \ref{M_prop_gen_JS} we conclude that $u \leq w$ in $\partial B_a(y_0)\cap\W$, so\footnote{Observe that no other assumption that the connectedness of $\partial B_a(y_0) \cap \W$ was required in step 2.}
\begin{equation}\label{M_desigualdade_bordo}
\ds\sup_{\partial B_a(y_0)\cap \W} u \leq \ds\sup_{\partial\W\setminus B_a(y_0)} u + \psi(a).
\end{equation}
Using \eqref{M_desigualdade_bordo} in \eqref{est_u0_1} from step 1 one gets 
$$u(y_0) \leq \ds\sup_{\partial\W\setminus B_a(y_0)} u  + \psi(a) +\sqrt{\dfrac{2a}{\nu}}.$$

Let us define $\psi$ by (see also \cite[\S 14.4]{GT})
\begin{equation}\label{M_expres_psi}
\psi(t)=\left(\dfrac{2}{n-1}\right)^{1/2}\ds\int_t^{\delta} \left(\log \frac{r}{a}\right)^{-1/2}dr.
\end{equation}
Such a function satisfies P5--P8, and also $\dfrac{n-1}{t} \psi'(t)^3  + \psi''(t)<0$ for each $t\in(a,\delta)$.
{
In fact,
$$\psi'(t)=-\left(\dfrac{2}{n-1}\right)^{1/2}\left(\log \frac{t}{a}\right)^{-1/2}$$
and
\begin{align*}
\psi''(t)=&-\left(\dfrac{2}{n-1}\right)^{1/2}\left(-\dfrac{1}{2}\left(\log \frac{t}{a}\right)^{-3/2}\dfrac{a}{t}\dfrac{1}{a}\right)\\
         =&\dfrac{1}{2t}\left(\dfrac{2}{n-1}\right)^{1/2}\left(\log \frac{t}{a}\right)^{-3/2}\\
				 =&\dfrac{n-1}{4t}\left(\dfrac{2}{n-1}\right)^{3/2}\left(\log \frac{t}{a}\right)^{-3/2}\\
				 =&-\dfrac{n-1}{4t}\psi'(t)^3\\
				 <&-\dfrac{n-1}{t}\psi'(t)^3.
\end{align*}
}

Additionally, it is easy to see that $\ds\lim_{a\rightarrow 0}\psi(a)=0$.
Hence, for each $\varepsilon>0$, $a$ can be chosen small enough to satisfy
\[
\psi(a)+\sqrt{\dfrac{2a}{\nu}}<\varepsilon.\qedhere
\]
\end{proof}

Now we are able to prove the following non-existence result. 
\begin{teo}[{\cite[Th. 3.4.5 p. 112]{han2016nonlinear}}]\label{M_nao_exist_MxR_Hxz}
Let $\W\subset M$ be a bounded domain whose boundary is of class $\cl^2$. Let $H\in\cl^0(\overline{\W})$ be a function either non-positive or non-negative. Assume that there exists $y_0\in\partial\W$ such that
$$(n-1) \Hc_{\partial\W}(y_0) < n \modulo{H(y_0)}.$$
Then, for any $\varepsilon>0$, there exists $\varphi\in\cl^{\infty}(\overline{\W})$ with $\modulo{\varphi}<\varepsilon$ on $\partial\W$, such that there exists no $u \in \cl^2 (\Omega)\cap \cl^0(\overline{\Omega})$ satisfying problem \eqref{ProblemaP}. 
\end{teo}

\begin{proof}
Obviously it can be supposed that $H\geq 0$. 
For any $\varepsilon>0$ take $a$ as in the previous lemma. Let $\varphi\in\cl^{\infty}(\overline{\W})$ such that $\varphi=0$ in $\partial\W\setminus B_a(y_0)$, $0\leq \varphi \leq \varepsilon$ on $\partial\W\cap B_a(y_0)$ and $\varphi(y_0)=\varepsilon$. Hence, no solution of equation \eqref{operador_minimo_1_coord} in $\W$ could have $\varphi$ as boundary values because such a function does not satisfy \eqref{est_nao_exist_Hxz}. 
\end{proof}

%
%

\section{Prescribed mean curvature equations in Riemannian manifolds}

Dirichlet problems for equations whose solutions describe hypersurfaces of prescribed mean curvature have been also studied outside of the Euclidean space.  However, Serrin type solvability criteria have been obtained only in {few} cases. 

For instance, P.-A Nitsche \cite{Nitsche2002} was concerned with graph-like prescribed mean curvature hypersurfaces in hyperbolic space $\HH^{n+1}$. In the half-space setting, he studied radial graphs over the totally geodesic hypersurface $S = \{x \in \R^{n+1}_+; (x_0)^2 + \dots + (x_n)^2 = 1\}$. He established an existence result if $\W$ is a bounded domain of $S$ of class $\cl^{2,\alpha}$ and $H\in\cl^1(\overline{\W})$ is a function satisfying $\sup\limits_{\overline{\W}}\modulo{H}\leq 1$ and $\modulo{H(y)}<\Hc_{C}(y)$ everywhere on $\partial\W$, where $\Hc_{C}$ denotes the hyperbolic mean curvature of the cylinder $C$ over $\partial\W$. Furthermore, he showed the existence of smooth boundary data such that no solution exists in case of $\modulo{H(y)}>\Hc_{C}(y)$ for some $y\in\partial\W$ under the assumption that $H$ has a sign. We observe that these results do not provide Serrin type solvability criterion.

Also in the half space model of the hyperbolic space, E. M. Guio-R. Sa Earp \cite{elias,eliasarticle} considered a bounded domain $\W$ contained in a vertical totally geodesic hyperplane $P$ of $\HH^{n+1}$ and studied the Dirichlet problem for the mean curvature equation for horizontal graphs over $\W$, that is, hypersurfaces which intersect at most only once the horizontal horocycles orthogonal to $\W$. They considered the hyperbolic cylinder $C$ generated by horocycles cutting ortogonally $P$ along the boundary of $\W$ and the Serrin condition, $\Hc_C(y) \geq \modulo{H(y)}$ $\forall\ y\in\partial\W$. They obtained a Serrin type solvability criterion for prescribed mean curvature $H=H(x)$ and also proved a sharp solvability criterion for constant $H$.

There are also some results of this type in the Riemannian product $M\times\R$, where $M$ is a complete Riemannian manifold of dimension $n\geq 2$. Analogously to the Euclidean setting, the solutions of the equation 
\begin{equation}\label{operador_minimo_1_manifolds}
\diver_M\left(\dfrac{\nabla_M u}{\sqrt{1+\norm{\nabla_M u}_M^2}}\right) = nH
\end{equation}
are vertical graphs in $M\times\R$ with mean curvature $H$ at each point of the graph.
However, even though the study of the Dirichlet problem for equation \eqref{operador_minimo_1_manifolds} inherits the techniques from the Euclidean setting, it is more difficult. 
For instance, in a coordinate system $(x_1,\dots,x_n)$ in $M$, the non-divergence form of equation \eqref{operador_minimo_1_manifolds} is equivalent to 
\begin{equation}\label{operador_minimo_1_manifolds_coord}
\M u:=\sum_{i,j=1}^n \left(W^2\sigma^{ij} - {u^iu^j} \right)\Hessij u=nH{W^3},
\end{equation}
where $(\sigma^{ij})$ is the inverse of the metric $(\sigma_{ij})$ of $M$, $u^i=\ds\sum_{j=1}^n\sigma^{ij} \Dj u$ are the coordinates of $\nabla u$ and $\Hessij u(x)=\Hess u(x){\left(\Ei,\Ej\right)}$.

In this context Aiolfi-Ripoll-Soret {\cite[Th. 1 p. 72]{Aiolfi}} proved that there always exists a vertical minimal graph ($H=0$) in $M\times\R$ over a mean convex, smooth and bounded domain $\W$ in $M$ for arbitrary continuous boundary data.  This result generalizes the existence part in Theorem \ref{SharpJenkinsSerrin} stated in the introduction. 
In the case where $M$ is a Hadamard manifold whose sectional curvature is bounded above by $-1$, then the mean convexity condition is sharp due to a work of M. Telichevesky {\cite[Th. 6 p. 246]{miriam}}. 
The combination of these two results gives a sharp solvability criterion for the minimal hypersurface equation in bounded domains of these types of Hadamard manifolds.

{The author of these notes have generalized the aforementioned non-existence result in the $M\times\R$ context on her PhD thesis \cite{minhatese} supervised by professor R. Sa Earp}. More precisely, it was proved that if $H$ is a continuous function and non-decreasing in the variable $z$, then the \textit{strong Serrin condition}
\begin{equation}\label{SerrinConditionGeneral}
(n-1)\Hc_{\partial\W}(y)\geq n\ds\sup_{z\in\R}\modulo{H(y,z)} \ \forall \ y\in\partial\W
\end{equation}
is necessary for the solvability of the Dirichlet problem 
\begin{equation}\tag{$P_{M\times\R}$}\label{ProblemaP_manifolds}
\left\{
\begin{split}
\M u &=n H(x,u)W^3 \ \mbox{in}\ \W,\\
u&=\varphi \ \mbox{in}\ \partial\W.
\end{split}\right.
\end{equation} 
in every Hadamard manifold {\cite[Th. 2.5 p. 26]{minhatese} (see also \cite[Cor. 2 p. 3]{AlvarezNonExistence})} and in every compact and simply connected manifold which is strictly $1/4-$pinched\footnote{A Riemannian manifold is said to be strictly $1/4-$pinched if the sectional curvature $K$ of $M$ satisfies $\frac{1}{4} K_0 < K \leq K_0$ for a positive constant $K_0$.} {\cite[Th. 2.6 p. 27]{minhatese} (see also \cite[Cor. 3 p. 4]{AlvarezNonExistence})}.

Some direct consequences derived from these non-existence results are the following. Firstly, the combination of the non-existence result for Hadamard manifolds \cite[Th. 2.5 p. 26]{minhatese} 
with the existence theorem from Aiolfi-Ripoll-Soret {\cite[Th. 1 p. 72]{Aiolfi}} for the minimal case shows that the sharp solvability criterion of Jenkins-Serrin (see Theorem \ref{SharpJenkinsSerrin} stated in the introduction) actually holds in every Cartan-Hadamard manifolds:
\begin{teo}[Sharp Jenkins-Serrin-type solvability criterion]\label{cond_nece_suf_minimo_Hadamard}
Let $M$ be a Cartan-Hadamard manifold and $\W\subset M$ a bounded domain whose boundary is of class $\cl^{2,\alpha}$ for some $\alpha\in(0,1)$. Then the Dirichlet problem for the minimal equation in $\W$ has a unique solution for arbitrary continuous boundary data if, and only if, $\W$ is {mean convex}.
\end{teo}
Secondly, combining the non-existence result for positively curved manifolds \cite[Th. 2.6 p. 27]{minhatese} 
with an existence result of Spruck {\cite[Th. 1.4 p. 787]{spruck}} we infer the following:
\begin{teo}[Sharp Serrin-type solvability criterion]\label{cond_nece_suf_minimo_curv_positivo_este_este}
Let $M$ be a compact and simply connected manifold which is strictly $1/4-$pinched. Let $\W\subset M$ be a domain with $\diam(\W)<\frac{\pi}{2\sqrt{K_0}}$ and whose boundary is of class $\cl^{2,\alpha}$ for some $\alpha\in(0,1)$. Then for every constant $H$ the Dirichlet problem \eqref{ProblemaP_manifolds} has a unique solution for arbitrary continuous boundary data if, and only if, $(n-1)\Hc_{\partial\W}\geq n\modulo{H}$.
\end{teo}

Notice that it was not derived directly 
a sharp Serrin type result (see Theorem \ref{SharpSerrin} stated in the introduction) for arbitrary constant $H$ in every Cartan-Hadamard manifold. This is due to the fact that in Hadamard manifolds we have some geometric restrictions. For example, not in every mean convex domain of a Hadamard manifolds the mean curvature of the parallel hypersurfaces is increasing along the geodesics normals to $\partial\W$. 

By way of illustrating better this fact let $M=\HH^n$. It follows from the existence result of Spruck \cite[Th. 1.4 p. 787]{spruck} that the Serrin condition is a sufficient condition if $H\geq \frac{n-1}{n}$. In the opposite case $0<H<\frac{n-1}{n}$, Spruck noted the existence of an entire graph of constant mean curvature $\frac{n-1}{n}$ in $\HH^n\times\R$ (see \cite{BerardRicardo} for explicit formulas) whose vertical translations and reflexions are barrier for the solutions of the problem. Having this height estimate it was possible to establish an a priori boundary gradient estimate if the strict inequality $(n-1)\Hc_{\partial\W} > n {H}$ holds {since in this case there exists a tubular neighborhood of $\partial\W$ on which $(n-1)\Hc_{\Gamma_t} > n {H}$ for every hypersurface parallel to $\partial\W$ contained on it.} This restriction over the Serrin condition in the last case did not allows us to establish a Serrin type solvability criterion for every constant $H$ directly from the combination of the existence result of Spruck {\cite[Th. 5.4 p. 797]{spruck}} with our non-existence result for Hadamard manifods {\cite[Th. 2.5 p. 26]{minhatese}} when the ambient is the hyperbolic space.

We have also established an existence result {\cite[Th. 4.4 p. 51]{minhatese} (see also \cite[Th. 5 p. 4]{AlvarezExistence})} for prescribed $H\in\cl^{1,\alpha}(\overline{\W}\times\R)$ which extends the existence result of Spruck and that yields the following Serrin type solvability criterion when combined with the non-existence theorem for Hadamard manifold {\cite[Th. 2.5 p. 26]{minhatese}} mentioned before:
\begin{teo}[Serrin type solvability criterion in $\HH^n\times\R$]\label{cond_nece_suf_HnxR}
Let $\W\subset \HH^n$ be a bounded domain with $\partial\W$ of class $\cl^{2,\alpha}$ for some $\alpha\in(0,1)$. Let $H\in\cl^{1,\alpha}(\overline{\W}\times\R)$ be a function satisfying $\Dz H\geq 0$ and $0 \leq {H} \leq \frac{n-1}{n}$ in ${\W\times\R}$. Then the Dirichlet problem \eqref{ProblemaP_manifolds} has a unique solution $u\in\cl^{2,\alpha}(\overline{\W})$ for every $\varphi\in\cl^{2,\alpha}(\overline\W)$ if, and only if, the strong Serrin condition \eqref{SerrinConditionGeneral} holds.
\end{teo}

Combining Theorem \ref{cond_nece_suf_HnxR} with the non-existence result for Hadamard manifolds {\cite[Th. 2.5 p. 26]{minhatese}} and the existence result of Spruck {\cite[Th. 1.4 p. 787]{spruck}} we deduce that the sharp solvability criterion of Serrin {(Theorem \ref{SharpSerrin})} also holds in the $\cl^{2,\alpha}$ class if we replace $\R^n$ by $\HH^n$:
\begin{teo}[Sharp Serrin type solvability criterion in $\HH^n\times\R$]\label{cond_nece_suf_HnxR_constate}
Let $\W\subset \HH^n$ be a bounded domain whose boundary is of class {$\cl^{2,\alpha}$}. Then for every constant $H$ the Dirichlet problem \eqref{ProblemaP_manifolds} has a unique solution for arbitrary continuous boundary data if, and only if, $(n-1) \Hc_{\partial\W}(y) \geq n \modulo{H}$.
\end{teo}

We have also proved \cite[Th. 4.1 p. 40]{minhatese} (see also \cite[Th. 4 p. 4]{AlvarezExistence}) a generalization of the Spruck's existence result {\cite[Th. 1.4 p. 787]{spruck}} for constant mean curvature.  
Putting together this result with the non-existence theorem for Hadamard manifolds {\cite[Th. 2.5 p 26]{minhatese}} we derive the following generalization in the $\cl^{2,\alpha}$ class of Theorem \ref{T_Serrin_Ricci} of Serrin stated in the introduction:
\begin{teo}[Serrin type solvability criterion 2]\label{cond_nece_suf_hadamard}
Let $M$ be a Cartan-Hadamard manifold and $\W\subset M$ a bounded domain whose boundary is of class $\cl^{2,\alpha}$ for some $\alpha\in(0,1)$. Suppose that $H\in\cl^{1,\alpha}(\overline{\W}\times\R)$ is either non-negative or non-positive in $\overline{\W}\times\R$, $\Dz H\geq 0$ and
$$\Ricc_x\geq n\ds\sup_{z\in\R}\norm{\nabla_x H(x,z)}-\dfrac{n^2}{n-1}\ds\inf_{z\in\R}\left(H(x,z)\right)^2, \ \forall \ x\in\W.$$
Then the Dirichlet problem \eqref{ProblemaP_manifolds} has a unique solution $u\in\cl^{2,\alpha}(\overline{\W})$ for every $\varphi\in\cl^{2,\alpha}(\overline{\W})$ if, and only if, the strong Serrin condition \eqref{SerrinConditionGeneral} holds.
\end{teo}

Finally, using the non-existence result for the types of positively curved manifold mentioned above we also derive:
\begin{teo}[Serrin type solvability criterion 3]
\label{cond_nece_suf_curv_posit}
Let $M$ be a compact and simply connected manifold which is strictly $1/4-$pinched. Let $\W\subset M$ be a domain with $\diam(\W)<\frac{\pi}{2\sqrt{K_0}}$ and whose boundary is of class $\cl^{2,\alpha}$ for some $\alpha\in(0,1)$. Suppose that $H\in\cl^{1,\alpha}(\overline{\W}\times\R)$ is either non-negative or non-positive in $\overline{\W}\times\R$, $\Dz H\geq 0$ and
$$\Ricc_x\geq n\ds\sup_{z\in\R}\norm{\nabla_x H(x,z)}-\dfrac{n^2}{n-1}\ds\inf_{z\in\R}\left(H(x,z)\right)^2, \ \forall \ x\in\W.$$
Then the Dirichlet problem \eqref{ProblemaP_manifolds} has a unique solution $u\in\cl^{2,\alpha}(\overline{\W})$ for every $\varphi\in\cl^{2,\alpha}(\overline{\W})$ if, and only if, the strong Serrin condition \eqref{SerrinConditionGeneral} holds.
\end{teo}

Other works have considered a Serrin type condition that provides some existence theorems in more general context (see \cite{Alias2008}, \cite{Dajczer2008}, \cite{Dajczer2005} as examples). However, to the best of our knowledge, no other Serrin-type solvability criterion has been proved in settings different from the Euclidean one.

\phantomsection
\addcontentsline{toc}{section}{Bibliography}
\bibliographystyle{amsplain}
\bibliography{bibliografia}

\end{document}